\documentclass{amsart} 

\usepackage{xcolor}

\usepackage{amsrefs,amsmath,amsthm,amssymb}            

\usepackage{enumerate}
\usepackage{lscape}

\usepackage[matrix,arrow,curve,frame]{xy}
\usepackage{longtable,booktabs}

\definecolor{mblue}{RGB}{65,105,225}
\usepackage[colorlinks=true,linkcolor=blue,citecolor=mblue]{hyperref}

\newtheorem{thm}[subsection]{Theorem}
\newtheorem{prop}[subsection]{Proposition}
\newtheorem{cor}[subsection]{Corollary}
\newtheorem{lemma}[subsection]{Lemma}

\theoremstyle{definition}  
\newtheorem{defn}[subsection]{Definition}
\newtheorem{ex}[subsection]{Example}
\newtheorem{remark}[subsection]{Remark}
\newtheorem{strategy}[subsection]{Strategy}

\newcommand{\map}{\rightarrow}
\newcommand{\cl}{\mathrm{cl}}
\newcommand{\tp}{\mathrm{top}}
\newcommand{\sigmabar}{\overline{\sigma}}
\newcommand{\kappabar}{\overline{\kappa}}
\newcommand{\hsf}{\mathsf{h}}
\newcommand{\an}[1]{\left\langle {#1}\right\rangle}

\newcommand{\AR}{{\mathcal{A}}}
\newcommand{\AC}{{\mathcal{A}^\C_*}}
\newcommand{\Acl}{{\mathcal{A}^{\text{cl}}_*}}
\newcommand{\tensor}{\otimes}

\newcommand{\field}[1]  {\mathbb #1} 
\newcommand{\F}         {\field F}
\newcommand{\R}         {\field R}
\newcommand{\Z}         {\field Z}
\newcommand{\C}         {\field C}
\newcommand{\M}         {\field M}
\newcommand{\A}         {\field A}

\DeclareMathOperator{\Sq}{Sq}
\DeclareMathOperator{\Ext}{Ext}
\DeclareMathOperator{\coker}{coker}

\numberwithin{equation}{section} 

\begin{document}

\title{$\R$-motivic stable stems}

\author{Eva Belmont} 
\address{Department of Mathematics \\Northwestern University \\Evanston, IL
60208}
\email{ebelmont@northwestern.edu}

\author{Daniel C.\ Isaksen}
\address{Department of Mathematics \\ Wayne State University\\
Detroit, MI 48202}
\email{isaksen@wayne.edu}




\begin{abstract}
We
compute some $\R$-motivic stable homotopy groups.
For $s - w \leq 11$,
we describe the motivic stable homotopy groups
$\pi_{s,w}$ of a completion of the $\R$-motivic sphere spectrum.  
We apply the $\rho$-Bockstein spectral
sequence to obtain $\R$-motivic $\Ext$ groups 
from the $\C$-motivic $\Ext$ groups, which are well-understood in a large
range.
These $\Ext$ groups are the input to the
$\R$-motivic Adams spectral sequence.  
We fully analyze the Adams differentials in a range, and we
also analyze hidden extensions by $\rho$, $2$, and $\eta$.
As a consequence of our computations,
we recover Mahowald invariants of many low-dimensional
classical stable homotopy elements.
\end{abstract}

\maketitle

\section{Introduction}

The goal of this article is to compute the stable homotopy groups
of the $\R$-motivic sphere spectrum in a range.  These stable
homotopy groups are the most fundamental invariants of the
$\R$-motivic stable homotopy category, and thus lead to a deeper 
understanding of many of the computational aspects of
$\R$-motivic homotopy theory.
More specifically, we work in cellular $\R$-motivic stable homotopy
theory, completed appropriately at $2$ and $\eta$ so that the
$\R$-motivic Adams spectral sequence converges.  

Our main tool is the $\R$-motivic Adams spectral sequence, which 
takes the form
\[
E_2 = \Ext_{\AR}(\M_2, \M_2) \implies \pi_{**}.
\]
Here $\AR$ is the $\R$-motivic Steenrod algebra,
$\M_2$ is the $\R$-motivic cohomology of a point, and
$\pi_{*,*}$ is the bigraded homotopy groups of the
$\R$-motivic sphere (completed at $2$ and $\eta$).
We obtain complete results about
$\pi_{s,w}$ for $s - w \leq 11$.
This approach follows \cite{DI17}, which computed
$\pi_{s,w}$ for $s-w \leq 3$.

See \cite{BI20} for large-scale $\R$-motivic Adams charts.
These charts are an essential companion to this manuscript.
In a sense, this manuscript consists of a series of arguments for
the computational facts displayed in the Adams charts.

\subsection{The $\rho$-Bockstein spectral sequence}
The first step in an Adams spectral sequence program is to obtain
the algebraic $E_2$-page.  We study this computation in 
Sections \ref{sctn:rho-Bockstein}, \ref{sctn:Bock-diff}, and
\ref{section:bockstein-extn}.
We use the $\rho$-Bockstein spectral sequence, which takes
the form
\[
\Ext_{\AR^\C}(\M_2^\C, \M_2^\C) [ \rho ] \implies
\Ext_{\AR}(\M_2, \M_2).
\]
Here $\AR^\C$ is the $\C$-motivic Steenrod algebra,
and $\M_2^\C$ is the $\C$-motivic cohomology of a point.

The $\rho$-Bockstein spectral sequence is a tool that passes from
$\C$-motivic $\Ext$ groups to $\R$-motivic $\Ext$ groups.
We discuss the general properties of this spectral sequence
in Section \ref{sctn:rho-Bockstein}, and we describe an 
unexpectedly effective strategy for computing differentials.
The key idea is to compute 
the $\rho$-periodic groups 
$\Ext_{\AR}(\M_2, \M_2)[\rho^{-1}]$ in advance.
Then naive combinatorial considerations force a very large number of
Bockstein differentials.
We discuss specific Bockstein differential computations in
Section \ref{sctn:Bock-diff}.

Having obtained the $E_\infty$-page of the $\rho$-Bockstein spectral
sequence, we do not yet have a complete knowledge of
$\Ext_{\AR}(\M_2, \M_2)$.  It remains to resolve extensions that are
hidden by the $\rho$-Bockstein filtration.
There is an unmanageable quantity of hidden extensions,
so we do not attempt to analyze them completely, not even in a range.
Nevertheless, we do analyze all extensions by $h_0$ and $h_1$ 
in the range under consideration.  These computations are carried out
in Section \ref{section:bockstein-extn}.

\subsection{The $\R$-motivic Adams spectral sequence}

Having obtained the $E_2$-page 
of the $\R$-motivic Adams spectral sequence,
the next step is to determine Adams differentials.
We carry out these computations in Section \ref{sctn:Adams-diff}.
These differentials can be obtained by a variety of techniques.
One important technique is the use of the Moss Convergence Theorem \ref{thm:Moss} to compute Toda brackets, which determine that
certain elements are permanent cycles.
Another technique is comparison to previously
established computations in the $\C$-motivic and classical computations.
See Section \ref{subsctn:intro-comparison} 
for more discussion of these comparisons.

After computing Adams differentials and obtaining the 
Adams $E_\infty$-page, there are once again hidden extensions to
resolve.  As in the algebraic case, there are too many extensions
to study exhaustively, but we do consider all extensions by
$\rho$, $\hsf$, and $\eta$ exhaustively (where $\rho$, $\hsf$, and
$\eta$ are stable homotopy elements detected by
$\rho$, $h_0$, and $h_1$ respectively).
These computations are carried out in Section \ref{sctn:Adams-extns}.
Once again, the key techniques are shuffling relations involving Toda brackets and comparison to the $\C$-motivic and classical cases.

\subsection{Comparison of homotopy theories}
\label{subsctn:intro-comparison}
An essential ingredient in our computations is comparison between
the $\R$-motivic, $\C$-motivic, $C_2$-equivariant, and classical
stable homotopy theories, as depicted in the diagram
\begin{equation}
\label{eq:comparison}
\xymatrix{
\R\textrm{-motivic} \ar[rr]^{\textrm{realization}} \ar[d]_{\textrm{extension of scalars}} & & C_2\textrm{-equivariant} \ar[d]^{\textrm{forgetful}} \\
\C\textrm{-motivic} \ar[rr]_{\textrm{realization}} & & \textrm{classical.}
}
\end{equation}
The horizontal arrows labelled ``realization" refer to
the Betti realization functors that 
take a variety over $\C$ (resp., over $\R$)
to the space (resp., $C_2$-equivariant space) of $\C$-valued points.
The vertical arrow labelled ``extension of scalars" refers
to the functor that takes a variety over $\R$ and views it as a
variety over $\C$.
The vertical arrow labelled ``forgetful" refers to the functor
that takes a $C_2$-equivariant object to its underlying non-equivariant
object.

Our philosophy in this article is to accept computational information
about the $\C$-motivic and classical stable homotopy groups
as given, and to use this information to study the
$\R$-motivic stable homotopy groups.
See \cite{Isaksen14c} for an extensive summary of
computational information about the
$\C$-motivic and classical Adams spectral sequences.
The presence of the $C_2$-equivariant stable homotopy category in
this diagram is relevant for our consideration of
Mahowald invariants, to be discussed below in 
Section \ref{subsctn:intro-root}.

There is a surprising connection between $\C$-motivic and $\R$-motivic 
that enables many of our detailed computations.
Namely, Theorem \ref{thm:pi-S/rho} shows that the
$\C$-motivic stable homotopy groups are isomorphic to the
$\R$-motivic homotopy groups of the cofiber $S/\rho$ of $\rho$.
This means that the structure of $\C$-motivic stable homotopy groups
governs both the cokernel and the kernel of multiplication by $\rho$.
This allows us to deduce many $\R$-motivic computational facts
with relative ease from known $\C$-motivic information.

\subsection{Mahowald invariants}
\label{subsctn:intro-root}

Let $\alpha$ be a non-zero classical stable homotopy element.
The Mahowald invariant (or root invariant) $R(\alpha)$ is a non-zero equivalence class
of classical stable homotopy elements in a stem that is higher
than the stem of $\alpha$.
One source of interest in Mahowald invariants is that $R(\alpha)$
appears to have greater chromatic complexity than $\alpha$.
Thus one can construct more exotic stable homotopy elements
out of elements that are better understood
\cite{mahowald-ravenel-root}.

Bruner and Greenlees reformulated the definition of the Mahowald
invariant in terms of $C_2$-equivariant stable homotopy groups 
\cite{bredon-loffler}.  Although we do not study $C_2$-equivariant
homotopy groups directly, we have indirectly obtained information
about them because the 
$\R$-motivic and $C_2$-equivariant stable homotopy groups
are isomorphic in a range \cite{BGI19}.
In Section \ref{sctn:root}, we show how many Mahowald invariants
can be immediately deduced from our $\R$-motivic computations.
While these results only recover previously known Mahowald invariants
\cite{mahowald-ravenel-root} \cite{Behrens07},
we believe that our techniques can be extended into uncharted territory
without much more effort.

\begin{thm}
\label{thm:R-Mahowald}
Table \ref{tab:Mahowald} gives some values of the
Mahowald invariant.
\end{thm}

\begin{longtable}{llll}
\caption{Some Mahowald invariants} \\
\toprule 
stem & $\alpha$ & $R(\alpha)$ & 
	indeterminacy \\
\midrule \endhead
\bottomrule \endfoot
\label{tab:Mahowald}
$0$ & $2$ & $\eta$ & \\
$0$ & $4$ & $\eta^2$ & \\
$0$ & $8$ & $\eta^3$ & \\
$1$ & $\eta$ & $\nu$ & $2\nu$, $4 \nu$ \\
$2$ & $\eta^2$ & $\nu^2$ \\
$3$ & $\nu$ & $\sigma$ & $2 \sigma$, $4 \sigma$, $8 \sigma$ \\
$3$ & $2\nu$ & $\eta \sigma$ & $\epsilon$ \\
$3$ & $4 \nu$ & $\eta^2 \sigma$ & $\eta \epsilon$ \\
$6$ & $\nu^2$ & $\sigma^2$ & $\kappa$ \\
$7$ & $\sigma$ & $\sigma^2$ \\
$7$ & $2 \sigma$ & $\eta_4$ & $\eta \rho_{15}$ \\
$7$ & $4 \sigma$ & $\eta \eta_4$ & $\nu \kappa$, $\eta^2 \rho_{15}$ \\
$8$ & $\eta \sigma$ & $\nu_4$ & $2 \nu_4$, $4 \nu_4$ \\
$8$ & $\epsilon$ & $\sigmabar$ \\
$9$ & $\eta^2 \sigma$ & $\nu \nu_4$ & $\eta \kappabar$ \\
\end{longtable}

\begin{proof}
Theorem \ref{thm:root-compare} reduces the computation
to an $\R$-motivic Mahowald invariant, as defined in 
Section \ref{subsctn:R-root}.
Table \ref{tab:R-Mahowald} gives the
values of the $\R$-motivic Mahowald invariant.  Finally,
Table \ref{tab:R-to-C} gives the Betti realizations of the
$\R$-motivic Mahowald invariants.
\end{proof}

See Examples \ref{ex:MR(sigma)} and \ref{ex:R(sigma)} for detailed
illustrations of how this technique plays out in practice.

We have computed the Mahowald invariant of most, but not every, 
$\alpha$ through the 11-stem.  In particular, we do not compute 
the Mahowald invariants of $2^k$ for $k \geq 4$,
$8 \sigma$, $\eta \epsilon$, $\mu_9$, $\eta \mu_9$,
nor $\zeta_{11}$ and its multiples.
In these cases,
the problem is that the inequality of Theorem \ref{thm:root-compare}
does not apply, so our $\R$-motivic computations do not determine
$C_2$-equivariant behavior.

\section{Notation}

We write $\M_2$ for the $\R$-motivic homology of a point
with coefficients in $\F_2$.  Recall that $\M_2$ is isomorphic to
$\F_2[\rho, \tau]$, where $\rho$ and $\tau$ have degrees $(-1,-1)$ and
$(0,-1)$ respectively \cite{Voevodsky03b}.

We write $\AR$ for the $\R$-motivic dual Steenrod algebra.
Recall that $\AR$ is described by the equations
\begin{align*}
\\\AR  & = \M_2[\tau_0,\tau_1,\dots,\xi_1,\xi_2,\dots]/(\tau_k^2=\tau
\xi_{k+1}+\rho \tau_{k+1}+ \rho \tau_0 \xi_{k+1})
\\\eta_L(\tau) & = \tau, \hspace{10pt} \eta_R(\tau) =\tau+\rho \tau_0,
\hspace{10pt} \eta_L(\rho)=\eta_R(\rho)=\rho
\\\Delta(\tau_k) & = \tau_k\tensor 1 + \sum {\xi}_{k-i}^{2^i}\tensor \tau_i
\\\Delta(\xi_k) & = \sum \xi_{k-i}^{2^i}\tensor \xi_i,
\end{align*}
where $\tau_i$ and $\xi_k$ have degrees
$(2^{i+1}-1, 2^i-1)$ and $(2^{i+1}-2,2^i-1)$ respectively
\cite{Voevodsky10}.

We write $\M_2^\C$ for the $\C$-motivic homology of a point
with coefficients in $\F_2$, and we write
$\AC$ for the $\C$-motivic dual Steenrod algebra.
These objects are easily described in terms of
$\M_2$ and $\AR$.  Namely, they are the result of setting
$\rho$ equal to zero.

We write $\Acl$ for the classical dual Steenrod algebra,
which can be obtained from $\AR$ by setting $\rho$ and $\tau$
to be $0$ and $1$ respectively.

We write $\Ext$ or $\Ext_\R$ for $\Ext_{\mathcal{A}}(\M_2, \M_2)$, i.e.,
the cohomology of the $\R$-motivic Steenrod algebra.
We write $\Ext_\C$ and $\Ext_\cl$ for the
cohomologies of the $\C$-motivic and classical Steenrod algebras
respectively.

We write $\pi_{p,q}$ or $\pi_{p,q}^\R$ for the
stable homotopy groups of the $\R$-motivic sphere spectrum.
Similarly, we write
$\pi_{p,q}^\C$ for the stable homotopy groups of the
$\C$-motivic sphere spectrum.
We adopt the usual motivic grading convention, so that
$\pi_{p,q}X$ denotes maps out of $S^{p,q}$, 
where $S^{p,q}$ is the smash product of $p-q$ copies of
the simplicial sphere and $q$ copies of $\A^1-0$.

We write $\pi^{C_2}_{p,q}$ for the
stable homotopy groups of the $C_2$-equivariant sphere spectrum.
We use an equivariant grading convention that is compatible with
the motivic grading convention, so that
$\pi_{p,q} X$ denotes maps out of $S^{p,q}$, 
where $S^{p,q}$ is the one-point compactification of $\R^p$,
with $C_2$ acting by negating the last $q$ coordinates.
Betti realization takes $\R$-motivic $S^{p,q}$ to
$C_2$-equivariant $S^{p,q}$.

We write $\pi_p$ for the classical stable homotopy groups.

All stable homotopy groups are suitably completed so that
Adams spectral sequences converge.  Classically, this means
completion at $2$.  In the motivic cases, this means completion
at $2$ and $\eta$ \cite{HKO11a}.

\subsection*{Grading conventions} 

Following \cite{Isaksen14c} and
\cite{DI17}, we use the following grading convention for the motivic Adams
spectral sequence: $s$ denotes the stem, $f$ denotes the Adams filtration, and $w$
denotes the motivic weight. 
Then the internal degree is $s+f$.
In this grading, Adams differentials take the form
$$ d_r:E_r^{s,f,w} \to E_r^{s-1,f+r,w}. $$

The \emph{coweight} of an element in degree $(s, f, w)$
is defined to be $s-w$.  Note that $\rho$
has coweight 0.  In particular, 
an element $x$ and its $\rho$-multiple $\rho x$ lie in the same
coweight.  This makes coweights particularly
useful in the $\rho$-Bockstein perspective that we adopt.

\subsection{Stable homotopy elements}
\label{subsctn:pi-notation}

We adopt conventional notation, as used (for example) in \cite{Isaksen14c}
\cite{IWX19}, for the names of elements in the classical
stable homotopy groups $\pi_*$ and the $\C$-motivic stable homotopy groups
$\pi^\C_{*,*}$.

Table \ref{tab:pi-notation} gives the notation that we use for 
elements of $\pi^\R_{*,*}$.
We define these elements in terms of the elements of the Adams
$E_\infty$-page that detect them.  These definitions have
indeterminacy parametrized by elements of the 
Adams $E_\infty$-page in higher Adams filtration.
As a general rule, this indeterminacy does not matter to our 
computations.  It is possible to use Toda brackets, or geometric
constructions (see \cite{DI13}), to eliminate the indeterminacy
in many cases.

\begin{remark}
\label{rem:hsf}
We use the symbol $\hsf$ to denote an element of $\pi_{0,0}$ that
is detected by $h_0$.  The symbol stands for ``hyperbolic" because
it corresponds to the hyperbolic plane in the Grothendieck-Witt
group interpretation of $\pi_{0,0}$ \cite{Morel04}*{Remark 6.4.2}.  
(Alternatively,
it can also stand for ``Hopf", since $\hsf$ is the zeroth Hopf map.)
Beware that $\hsf$ does not equal $2$; in fact, $2 = \hsf + \rho \eta$.  
\end{remark}

\begin{remark}
\label{rem:sigma}
The element $\sigma$ requires more discussion.  We write $\sigma$
for an element of $\pi_{7,4}$ that is detected by $h_3$.
There are 256 possible choices for $\sigma$, because of the presence
of elements in higher Adams filtration.  One such element in higher
filtration is $\rho c_0$.
Lemma \ref{lem:t^2h2 * c0} shows that $\tau^2 h_2 \cdot \rho c_0$
equals $\rho^4 d_0$.
Therefore, some possible choices of $\sigma$ have the property that
$\tau^2 \nu \cdot \sigma$ is detected by $\rho^4 d_0$ in $\pi_{10,4}$,
while other possible choices of $\sigma$ have the property that
$\tau^2 \nu \cdot \sigma$ is zero.
(The elements $\tau h_1 \cdot \tau P h_1$ and 
$\rho h_1 \cdot \tau h_1 \cdot \tau P h_1$ are not relevant,
by comparison to $kq$ as in Remark \ref{rem:kq}.)

We will need to use the relation $\tau^2 \nu \cdot \sigma = 0$ in
later computations, so we must assume that our choice of $\sigma$
satisfies this condition.
\end{remark}

\begin{remark}
In some cases, we have chosen names for elements of $\pi^\R_{*,*}$
that reflect the values of the extension of scalars
functor given in Table \ref{tab:R-to-C}.  For example,
we write $\tau \sigma^2$ for an element of $\pi^\R_{14,7}$ that
is detected by $\rho h_4$, since this element maps to
$\tau \sigma^2$ in $\pi^\C_{14,7}$.
\end{remark}

\begin{remark}
Beware that our use of the symbol $\kappabar$ is inconsistent with
its usage in \cite{Isaksen14c}.  In this manuscript,
$\tau \kappabar$ refers to a non-zero element of $\pi^\C_{20,11}$
that is detected by $\tau g$.
The symbol $\kappabar$ is used in \cite{Isaksen14c} for the same element.
\end{remark}

\begin{remark}
Occasionally we refer to stable homotopy elements that have no 
standard name.  In these cases, we use the symbol $\{ x\}$ to
indicate a stable homotopy element that is detected by an element
$x$ of an Adams $E_\infty$-page.
\end{remark}

\section{Comparison between $\R$-motivic and $\C$-motivic homotopy}
\label{sctn:compare-R-C}

We first discuss the relationship
between $\R$-motivic and $\C$-motivic stable homotopy theory.
We will use these ideas frequently in later sections
to obtain $\R$-motivic information from known $\C$-motivic information.

Consider the cofiber sequence
\[
\xymatrix@1{
S^{-1,-1} \ar[r]^\rho & S^{0,0} \ar[r] & S/\rho.
}
\]
The cofiber $S/\rho$ of $\rho$ is a $2$-cell complex whose
structure governs multiplication by $\rho$ in
the $\R$-motivic stable homotopy groups, in a sense to be made
precise in this section.
In addition, we will draw an unexpected connection between 
the $\R$-motivic homotopy groups of $S/\rho$ and
$\C$-motivic stable homotopy groups.

As shown in diagram (\ref{eq:comparison}), 
there is an extension of scalars functor
from $\R$-motivic stable homotopy theory to
$\C$-motivic stable homotopy theory, and a
Betti realization functor from $\C$-motivic stable homotopy 
theory to classical stable homotopy theory.
These functors take Eilenberg-Mac Lane spectra to Eilenberg-Mac Lane
spectra, and thus interact nicely with Adams spectral sequences.
In particular, they induce highly structured morphisms of Adams
spectral sequences.  We will frequently use these comparison
functors to deduce information about the $\R$-motivic Adams spectral
sequence from already known information about the $\C$-motivic
and classical Adams spectral sequences.
See \cite{Isaksen14c} for an extensive summary of
computational information about the
$\C$-motivic and classical Adams spectral sequences.

Extension of scalars takes the element $\rho$ of $\pi_{-1,-1}$ to zero.
In particular, it induces the map
$\M_2 \map \M_2^\C$ that takes $\rho$ to zero, and it similarly
induces the map $\AR \map \AC$ that takes $\rho$ to zero.

For an $\R$-motivic spectrum, we write $\Ext_\R(X)$ for
the $E_2$-page of the $\R$-motivic Adams spectral sequence
that converges to $\pi_{*,*}(X)$, i.e., for
$\Ext_{\mathcal{A}}(\M_2, H^{*,*}(X))$,
and similarly for $\Ext_\C(X)$.

Extension of scalars induces a diagram
\[
\xymatrix{
\ar[r] & \Ext_\R(S^{-1,-1}) \ar[r]^\rho \ar[d] & \Ext_\R(S^{0,0}) \ar[r] \ar[d] &
\Ext_\R(S/\rho) \ar[d] \ar[r] & \\ 
\ar[r] & \Ext_\C(S^{-1,-1}) \ar[r]^0 & \Ext_\C(S^{0,0}) \ar[r] &
\Ext_\C(S^{0,0} \vee S^{-2,-1}) \ar[r] & .
}
\]
Because $\rho$ becomes zero after extension of scalars, the bottom
row of the diagram splits.
The map $\Ext_\R(S/\rho) \map \Ext_\C(S^{0,0} \vee S^{-2,-1})$
lifts to a map $\Ext_\R(S/\rho) \map \Ext_\C(S^{0,0})$ that makes
the diagram
\[
\xymatrix{
\Ext_\R(S^{0,0}) \ar[r] \ar[d] & \Ext_\R(S/\rho) \ar[dl] \\
\Ext_\C(S^{0,0})
}
\]
commute.

\begin{prop}
\label{prop:Ext-S/rho}
The map $\Ext_\R(S/\rho) \map \Ext_\C(S^{0,0})$
is an isomorphism.
\end{prop}

\begin{proof}
Let $C^*_\R$ and $C^*_\C$ be the cobar complexes for
$\Ext_\R(S^{0,0})$ and $\Ext_\C(S^{0,0})$ respectively.
Note that $C^*_\C$ is isomorphic to $C^*_\R / \rho$.
Because multiplication by $\rho$ is injective on
$C^*_\R$,
this is also isomorphic to the cobar complex that computes
$\Ext_\R(S/\rho)$.
\end{proof}

\begin{remark}
\label{rem:ExtC-ExtR-module}
Because of the isomorphism of Proposition \ref{prop:Ext-S/rho},
the object $\Ext_\C$ is a module over $\Ext_\R$.
By careful inspection of definitions, this module action is
easy to describe.  Using the $\rho$-Bockstein spectral sequence
notation from Section \ref{sctn:rho-Bockstein}, a typical
element of $\Ext_\R$ is of the form $\rho^k x$, where
$x$ belongs to $\Ext_\C$.  
The $\Ext_\R$-module action on $\Ext_\C$ is described by
\[
\rho^k x \cdot y =
\left\{
\begin{array}{ll}
0 & \text{ if } k > 0 \\
xy & \text{ if } k = 0,
\end{array}
\right.
\]
where the last expression $xy$ is to be interpreted as the usual Yoneda product
of elements in $\Ext_\C$.
\end{remark}

\begin{remark}
\label{rem:Ext-R-C-exact}
Proposition \ref{prop:Ext-S/rho} implies that there is a long exact sequence
\[
\xymatrix@1{
\cdots \ar[r] & \Ext_\R \ar[r]^\rho & \Ext_\R \ar[r]^i & \Ext_\C \ar[r]^p &
\Ext_\R \ar[r]^\rho & \Ext_\R \ar[r] & \cdots
}
\]
of $\Ext_\R$-module maps, where $\Ext_\C$ is an $\Ext_\R$-module
as in Remark \ref{rem:ExtC-ExtR-module}.
If $x$ is a permanent cycle in the $\rho$-Bockstein spectral sequence,
then the map $i$ takes $x$ in $\Ext_\R$ to the element
of $\Ext_\C$ of the same name.
\end{remark}

Now consider the diagram
\begin{equation}
\label{eq:pi-S/rho}
\xymatrix{
\pi^\R_{*+1,*+1} \ar[r]^-\rho & \pi^\R_{*,*} \ar[r]\ar[d] & 
	\pi^\R_{*,*}(S/\rho)  \ar[dl] \\
& \pi^\C_{*,*}, }
\end{equation}
in which the diagonal arrow exists because $\rho$ maps to
zero in $\pi^\C_{*,*}$.

\begin{thm}
\label{thm:pi-S/rho}
The map $\pi^\R_{*,*}(S/\rho) \map \pi^\C_{*,*}$
is an isomorphism.
\end{thm}

\begin{proof}
Proposition \ref{prop:Ext-S/rho} shows that there is an isomorphism
of $E_2$-pages of Adams spectral sequences, so the targets
of the spectral sequences are also isomorphic.
\end{proof}

\begin{cor}
\label{cor:pi-rho}
Let $\alpha$ be an element of $\pi^\R_{*,*}$.
Extension of scalars takes $\alpha$ to zero in $\pi^\C_{*,*}$
if and only if
$\alpha$ is divisible by $\rho$.
\end{cor}
\begin{proof}
Chase the diagram \eqref{eq:pi-S/rho}, using that
the diagonal map is an isomorphism.
\end{proof}
\begin{remark}
Corollary \ref{cor:pi-rho} has a $C_2$-equivariant analogue,
as stated later in Proposition \ref{prop:C2-rho-divisible}.
\end{remark}

\begin{remark}
\label{rem:Behrens-Shah}
The isomorphism of Theorem \ref{thm:pi-S/rho} can be strengthened to an
equivalence of categories \cite{behrens-shah-C2}*{Corollary 8.6}.
Namely, the $2$-complete $\C$-motivic cellular stable homotopy category
is equivalent to the homotopy category of $S/\rho$-modules in the
$2$-complete $\R$-motivic cellular stable homotopy category.
\end{remark}

\begin{cor}\label{cor:LES}
There is a long exact sequence 
\[
\xymatrix@1{
\cdots \ar[r] & \pi_{s+1,w+1}^\R(S) \ar[r]^{\rho} & 
\pi_{s,w}^\R(S) \ar[r] &
\pi_{s,w}^\C(S) \ar[r] & \pi_{s,w+1}^\R(S) \ar[r] & \cdots.
}
\]
\end{cor}
\begin{proof}
This is the long exact sequence in homotopy for the fiber sequence 
\[
\xymatrix@1{
S \ar[r]^\rho & S \ar[r] &  S/\rho
}
\]
in $\R$-motivic spectra, after applying the identification in
Theorem \ref{thm:pi-S/rho}.
\end{proof}

\section{Mahowald invariants}
\label{sctn:root}

The goal of this section is to use $\R$-motivic computations
to recompute some Mahowald invariants.
See \cite{Behrens07}*{Section 4} for a careful discussion of the 
definition, using Lin's theorem that $\R P^\infty_{-\infty}$ is
equivalent to $S^{-1}$.

\subsection{$C_2$-equivariant homotopy theory and Mahowald invariants}

Using $C_2$-equivariant homotopy theory, 
Bruner and Greenlees \cite{bredon-loffler} gave an alternative
definition of the Mahowald invariant.
We will summarize this definition, but 
first we need some background on $C_2$-equivariant homotopy theory.

Let $S^{a,b}$ be the one-point compactification of $\R^a$,
where $C_2$ acts by negating the last $b$ coordinates.
Then $\rho: S^{0,0} \map S^{1,1}$ is the inclusion of fixed points.
Note that the cofiber of this map is $\Sigma (C_2)_+$, i.e.,
the suspension of the based free $C_2$-space.

We use the same notation $\rho$ for the
map $S^{-1,-1} \map S^{0,0}$ in the $C_2$-equivariant stable
homotopy group $\pi_{-1,-1}^{C_2}$.
The identification of the cofiber of $\rho$ leads immediately
to the following proposition, whose short proof appears in
\cite{GHIR20}*{Proposition 11.2}.

\begin{prop}
\label{prop:C2-rho-divisible}
Let $\alpha$ be a $C_2$-equivariant stable homotopy element.
The underlying classical stable homotopy element $U(\alpha)$ of
$\alpha$ is zero if and only if $\alpha$ is divisible by $\rho$.
\end{prop}

Geometric fixed points gives a map
$\pi_{a,b}^{C_2} \map \pi_{a-b}$, and this map takes $\rho$ to $1$.
The $\rho$-periodic groups $\pi_{*,*}^{C_2} [\rho^{-1}]$
are isomorphic to $\pi_* \otimes \Z[\rho^{\pm 1}]$,
i.e., to the classical stable homotopy groups with
$\rho$ and $\rho^{-1}$ adjoined 
\cite{Bredon67}*{Proposition}
\cite{AI82}*{Proposition 7.0}.

With this background on $C_2$-equivariant stable homotopy groups,
we now give the Bruner-Greenlees definition of the Mahowald invariant.
Start with a classical stable homotopy element $\alpha$
in $\pi_n$, which we identify with the obvious element 
of $\pi_* \otimes \Z[\rho^{\pm 1}]$ in degree $(0,-n)$.
Using the isomorphism
\[
\pi_* \otimes \Z[\rho^{\pm 1}] \cong
\pi^{C_2}_{*,*}[\rho^{-1}],
\]
write $\alpha = \rho^k \beta$ for some $\beta$ in
$\pi^{C_2}_{*,*}$ and some integer $k$, with $k$ maximal.
Finally, the Mahowald invariant $R(\alpha)$ is the 
underlying classical stable homotopy element $U(\beta)$ of
$\beta$.

Note that the Mahowald invariant is not strictly defined; 
it is a set of classical stable homotopy elements.  
While the choice of $k$ is unique,
the choice of $\beta$ is not.  Different choices of $\beta$ can
lead to different values of $U(\beta)$.

Also note that $U(\beta)$ is necessarily non-zero by 
Proposition \ref{prop:C2-rho-divisible}.  The point is that
$\beta$ is not divisible by $\rho$, since $k$ was chosen to be
maximal.

\subsection{$\R$-motivic homotopy theory and
Mahowald invariants}
\label{subsctn:R-root}

We will now adapt the framework of Bruner and Greenlees
\cite{bredon-loffler} from the $C_2$-equivariant to the
$\R$-motivic settings.  In order to carry this out, we need
to observe some key $\R$-motivic properties.

First, the $\rho$-periodic groups $\pi_{*,*}^{\R} [\rho^{-1}]$
are isomorphic to $\pi_* \otimes \Z[\rho^{\pm 1}]$,
i.e., to the classical stable homotopy groups with
$\rho$ and $\rho^{-1}$ adjoined 
\cite{DI17}.  See also \cite{Bachmann18} for a more structured
version of this isomorphism.
Second,
Corollary \ref{cor:pi-rho} relates $\rho$-divisibility to
the kernel of the extension of scalars map.

\begin{defn}
\label{defn:R-Mahowald}
Let $\alpha$ be a classical stable homotopy element in $\pi_n$.
The $\R$-motivic Mahowald invariant $R^\R(\alpha)$ is
defined as follows.
Identify $\alpha$ with the obvious element of 
\[
\pi_* \otimes \Z[\rho^{\pm 1}] \cong
\pi^\R_{*,*} [\rho^{-1}]
\]
in degree $(0,-n)$.
Write $\alpha = \rho^k \beta$ for some $\beta$ in
$\pi^\R_{*,*}$ and some integer $k$, with $k$ maximal.
Define $R^\R(\alpha)$ in $\pi^\C_{*,*}$ to be the extension
of scalars of $\beta$.
\end{defn}

\begin{remark}
\label{rem:Mahowald-indeterminacy}
As for the traditional Mahowald invariant, 
the $\R$-motivic Mahowald invariant is not strictly defined.
Different choices of $\beta$ can have different values 
in $\pi_{*,*}^\C$ under extension of scalars.
\end{remark}

\begin{remark}
\label{rem:Mahowald-nonzero}
As for the traditional Mahowald invariant,
the $\R$-motivic Mahowald invariant is always non-zero
by Corollary \ref{cor:pi-rho}.
The point is that
$\beta$ is not divisible by $\rho$, since $k$ was chosen to be
maximal.
\end{remark}

\begin{remark}
\label{rem:Quigley-comparison}
See \cite{Quigley18} \cite{Quigley19} for a different consideration of 
Mahowald invariants in the motivic context.  Our construction
does not compare directly.
\end{remark}

\begin{thm}\label{thm:motivic-mahowald-invariants}
Some values of the $\R$-motivic Mahowald invariant are given
in Table \ref{tab:R-Mahowald}.
\end{thm}

\begin{proof}
This follows immediately from the computations carried out later
in the article.  In particular, one needs the values of the 
extension of scalars map, as shown in Table \ref{tab:R-to-C}
and discussed in Section \ref{sctn:extn-scalars}
\end{proof}

\begin{ex}
\label{ex:MR(sigma)}
We illustrate Theorem \ref{thm:motivic-mahowald-invariants} by
describing the computation of $M^\R(\sigma)$.
The element $\sigma$ in $\pi_7$ is identified with the
element $\alpha$ of $\pi^\R_{*,*} \otimes \Z[\rho^{\pm 1}]$ in degree
$(0,-7)$ that is detected by $\rho^{15} h_4$.
Then $\alpha$ equals $\rho^{14} \beta$, where $\beta$ is detected
by $\rho h_4$.  Finally, Table \ref{tab:R-to-C} shows that
the realization of $\beta$ is $\tau \sigma^2$ in $\pi^\C_{14,7}$.
\end{ex}

In general, the relationship between $R(\alpha)$ and $R^{\R}(\alpha)$
is not obvious.  The choices involved in the definitions are not
necessarily compatible.
For example, it is possible that 
an element $\beta$ in $\pi^\R_{*,*}$ is not divisible by $\rho$,
while its realization in $\pi^{C_2}_{*,*}$ is divisible by 
$\rho$.

The main result of \cite{BGI19} tells us that 
the $\R$-motivic and $C_2$-equivariant stable homotopy groups
agree in a range.  In this range, $R(\alpha)$ and $R^\R(\alpha)$
are easier to compare.

\begin{thm}
\label{thm:root-compare}
Let $R^\R(\alpha)$ belong to $\pi^\C_{s,w}$, and 
Suppose that
$2w - s < 4$.
Then $R(\alpha)$ equals the Betti realization of
$R^\R(\alpha)$.
\end{thm}

\begin{proof}
The isomorphism between $\R$-motivic and $C_2$-equivariant stable
homotopy groups \cite{BGI19} implies that the choice of $\beta$
in the definition of $R^\R(\alpha)$ realizes to the
choice of $\beta$ in the definition of $R(\alpha)$.  
By the commutativity of the diagram (\ref{eq:comparison}),
the realization of $R^\R(\alpha)$ equals $R(\alpha)$.
\end{proof}

\begin{ex}
\label{ex:R(sigma)}
We showed in Example \ref{ex:MR(sigma)} that
$R^\R(\sigma)$ equals $\tau \sigma^2$ in $\pi^\C_{14,7}$.
The numerical condition of Theorem \ref{thm:root-compare} is
satisfied.
It follows that $R(\sigma)$ equals $\sigma^2$ in $\pi_{14}$, since
$\sigma^2$ is the realization of $\tau \sigma^2$.
\end{ex}

\begin{remark}
\label{rem:root-compare-failure}
Theorem \ref{thm:root-compare}, together with our 
computations of $\R$-motivic stable homotopy groups,
can be used to compute the Mahowald invariants $R(\alpha)$ for 
most $\alpha$ up to the $11$-stem.
The exceptions are $2^k$ for $k \geq 4$,
$8 \sigma$, $\eta \epsilon$, $\mu_9$, $\eta \mu_9$,
and $\zeta_{11}$ and its multiples.
In these cases, $R^\R(\alpha)$ can still be computed as shown
in Table \ref{tab:R-Mahowald}.  However, the numerical condition of
Theorem \ref{thm:root-compare} does not hold, so we cannot 
draw a conclusion about $R(\alpha)$ in these cases.
\end{remark}

\section{The $\rho$-Bockstein spectral sequence}
\label{sctn:rho-Bockstein}

We briefly recall some background on the $\rho$-Bockstein spectral
sequence that computes the cohomology of the
$\R$-motivic Steenrod algebra.  See \cite{Hill11} and \cite{DI17} for additional details.

Begin with the observation that 
the $\C$-motivic cohomology of a point $\M_2^\C$ equals
$\M_2 / \rho$, and the
$\C$-motivic dual Steenrod algebra $\AC$ equals
$\AR / \rho$.  Then filter the cobar complex by powers of $\rho$
to obtain 
the $\rho$-Bockstein spectral sequence
\begin{equation}
\label{rho-bockstein} 
E_1 = \Ext_{\AC}^{**}(\M_2^\C,
\M_2^\C)[\rho] \implies \Ext_{\AR}^{**}(\M_2,\M_2). 
\end{equation}

Our goal is to analyze the $\rho$-Bockstein spectral sequence
\eqref{rho-bockstein} in computational detail in a range of degrees.  
We recall some structural results
about this spectral sequence from \cite{DI17}.

\begin{prop}
\label{prop:rho-torsion}
\cite{DI17}*{Lemma 3.4}
If $d_r(x)$ is nontrivial in the $\rho$-Bockstein spectral
sequence, then $x$ and $d_r(x)$ are both $\rho$-torsion
free on the $E_r$-page.
\end{prop}

Recall that $\Acl$ is the classical dual Steenrod algebra.

\begin{prop}\cite{DI17}*{Theorem 4.1}
\label{prop:rho-local}
There is an isomorphism
$$ \Ext_{\Acl}(\F_2,\F_2)[\rho^{\pm 1}] \cong 
\Ext_\AR(\M_2,\M_2) [\rho^{-1}] $$
that takes elements
of degree $(s,f)$ in $\Ext_{\Acl}(\F_2,\F_2)$ to
elements of degree $(2s+f,f,s+f)$ in $\Ext_\AR(\M_2,\M_2)$.
In particular, the classical element $h_n$ corresponds
to the $\R$-motivic element $h_{n+1}$.
Moreover, the
isomorphism is highly structured, i.e., preserves products
and Massey products.
\end{prop}

The point of Proposition \ref{prop:rho-local} is that we 
a priori know the elements of $\Ext_\R$ that are 
$\rho$-periodic, in the sense that they support infinitely
many non-zero multiplications by $\rho$.  
In the range considered in this manuscript, these 
$\rho$-periodic elements are
$h_1$, $h_2$, $h_3$, $h_4$, $c_1$, $h_2 g$, $h_3 g$, as well as products of these elements.
This corresponds to the fact that through the $11$-stem,
$\Ext_{\cl}$
is generated by the classical elements
$h_0$, $h_1$, $h_2$, $h_3$, $c_0$, $P h_1$, and $P h_2$.
We may effectively ignore these $\rho$-periodic elements
when analyzing the $\rho$-Bockstein spectral sequence, since they 
can be neither source nor target of any 
$\rho$-Bockstein differential.

Let $\{ x_i \}$ be an $\F_2$-linear basis
for $\Ext_\C$, i.e., an $\F_2[\rho]$-linear basis
for the $\rho$-Bockstein $E_1$-page,
excluding the $\rho$-periodic
permanent cycles described in the previous paragraph.
For every $i$, either $x_i$ supports a differential, or
$\rho^r x_i$ is the target of the $d_r$ differential for some $r$.
In other words,
the set $\{ x_i \}$ may be partitioned into pairs $(x_i,x_j)$ such that
$d_r(x_i)=\rho^r x_j$ for some $j$.
Actually, one must be somewhat careful about the choice of basis
in situations where two or more elements of the basis have the same
degree.  Nevertheless, it is always possible to change basis so that
the basis elements can be partitioned into pairs.

The Bockstein differential 
$d_r:E_r^{s,f,w}\to E_r^{s-1,f+1,w}$ preserves the
quantity $s+f-w$, and $\rho$ lies in a degree satisfying $s+f-w = 0$.
Thus we may consider one value of $s+f-w$ at a time
when analyzing the $\rho$-Bockstein spectral sequence.

We exploit this structure in the following strategy for
analyzing the $\rho$-Bockstein spectral sequence.

\begin{strategy}
\label{strategy}
\mbox{}
\begin{enumerate}
\item
\label{step:N}
Fix a value $N = s+f-w$.
\item
Find an $\F_2[\rho]$-basis $B_N$ for
the part of the $\rho$-Bockstein $E_1$-page in degrees 
$(s,f,w)$ satisfying
$N = s+f-w$.
\label{step:basis}
\item
Remove elements from $B_N$
that detect $\rho$-periodic elements of $\Ext_\R$.
\item
\label{step:diff}
Use a variety of techniques, to be described below, to identify some
differential $d_r(x_i) = \rho^r x_j$, where $x_i$ and $x_j$
belong to $B_N$.
\item
\label{step:remove}
Remove $x_i$ and $x_j$ from $B_N$.
\item
Repeat steps (\ref{step:diff}) and (\ref{step:remove})
until $B_N$ is empty.
\end{enumerate}
\end{strategy}

For this strategy to be effective, we need to know that
the basis $B_N$ chosen in step \ref{step:basis} is finite.
Lemma \ref{lem:s+f-w-finite} establishes this fact.

\begin{lemma}
\label{lem:s+f-w-finite}
Let $N$ be fixed.  In degrees $(s,f,w)$ satisfying
$N = s+f-w$,
the $\rho$-Bockstein $E_1$-page is
a finitely generated $\F_2[\rho]$-module.
\end{lemma}

\begin{proof}
Recall that $\Ext_\C$ is non-zero only in degrees $(s,f,w)$ satisfying
$s+f-2w \geq 0$ 
\cite{Isaksen14c}*{Remark 2.20}.  
This inequality can be rewritten in the form
\[
s+f-w \geq \frac{1}{2}(s+f).
\]
In other words, we only need consider the part of
$\Ext_\C$ in total degree at most $2N$.
\end{proof}

One consequence of our strategy is that we 
do not compute the Bockstein differentials $d_r$ in order of 
increasing $r$.  Rather, we obtain all differentials as part of the
same process.

Step (\ref{step:diff}) is the limiting factor in the practical
effectiveness of our algorithm.  The 
ad hoc arguments required to establish specific differentials become
more difficult as the value of $N$ increases.  
However, these difficulties increase at a surprisingly slow rate,
and we are able to carry out the computation remarkably
far without much difficulty.

Our goal is to compute the $\rho$-Bockstein spectral
sequence through coweight $13$.
Unfortunately, infinitely many values of
$N$ in Step \ref{step:N} are relevant in this range.
For example, consider the elements $h_1^k$ of coweight $0$, which belong
to degrees satisfying $s+f-w = k$.

Similarly, any $h_1$-periodic sequence of elements
$h_1^k x$ of $\Ext_\C$ lies in degrees for which $s+f-w$ is
unbounded.  Fortunately, it is only these $h_1$-periodic
families that are problematic.

\begin{lemma}\label{lem:s+f-w}
Let $x$ be a non-zero element of $\Ext_\C$ of degree $(s,f,w)$
whose coweight is at most $k$.
Then:
\begin{enumerate}
\item
\label{case:h1-periodic}
$x$ is an $h_1$-periodic element, in the sense that 
$h_1^i x$ is non-zero for all $i \geq 0$; or
\item
$s+f-w\leq 3k+3$.
\end{enumerate}
\end{lemma}

\begin{proof}
If $2f - s \geq 4$,
then $x$ is $h_1$-periodic \cite{guillou-isaksen-1/2}.
So we may assume that
$2f - s < 4$.

By
\cite{Isaksen14c}*{Remark 2.20},
we also have the inequality
$s+f-2w\geq 0$. 
Combining with the assumption $s-w \leq k$, we conclude that
\[
s+f-w = (2f - s) - (s+f-2w) + 3(s - w) < 4 + 0 + 3k = 3k + 4.
\]
\end{proof}

As we wish to consider elements up to coweight $13$,
Lemma \ref{lem:s+f-w} suggests
we need to look at degrees satisfying the inequality $s+f-w\leq 42$,
in addition to studying $h_1$-periodic elements.
However, inspection of elements in $\Ext_\C$
shows that $s+f-w\leq 28$ for all elements that are relevant
in our range.

The $h_1$-periodic elements of $\Ext_\C$
are well-understood \cite{GI14}.
Up to coweight $13$,
all such elements are of the form
$1$, $P^k h_1$, $P^k c_0$, $P^k d_0$, $P^k e_0$,
$P^k c_0 d_0$, $d_0^2$, or $c_0 e_0$,
as well as the
$h_1$-multiples of these elements.
Lemma \ref{lem:s+f-w} indicates that the
behavior of the $\rho$-Bockstein spectral sequence on these elements
must be studied separately.
See Proposition \ref{prop:Bock-h1-local}
for the analysis of these $h_1$-periodic elements.

\section{$\rho$-Bockstein differentials}
\label{sctn:Bock-diff}

The goal of this section is to describe a variety of methods
for determining $\rho$-Bockstein differentials.  These methods
are applied in Step (\ref{step:diff}) of Strategy \ref{strategy}.
Taken together, these methods allow us to determine
all $\rho$-Bockstein differentials through
coweight $13$.

We begin with a result that describes all $\rho$-Bockstein
differentials on the elements of Adams filtration zero.

\begin{prop}
\label{prop:bock-tau^k}
\cite{DI17}*{Proposition 3.2}
\mbox{}
\begin{enumerate}
\item
$d_1(\tau) = \rho h_0$.
\item
$d_{2^k} (\tau^{2^k} ) = \rho^{2^k} \tau^{2^{k-1}} h_k$ 
for $k \geq 1$.
\end{enumerate}
\end{prop}

Next we consider $h_1$-periodic elements.  These elements must
be treated as special cases because of Case (\ref{case:h1-periodic})
of Lemma \ref{lem:s+f-w}.

\begin{prop}\label{prop:Bock-h1-local}
Table \ref{tab:Bock-h1-local} gives some 
Bockstein differentials that are non-zero after
inverting $h_1$. 
Through coweight $13$,
these are the only
$h_1$-periodic $\rho$-Bockstein differentials.
\end{prop}

For legibility, we have not included powers of $\rho$ in the
values of the Bockstein differentials in Table \ref{tab:Bock-h1-local}.
For example, the first row of the table is to be interpreted as
$d_3(P h_1) = \rho^3 h_1^3 c_0$.

\begin{proof}
The differentials in the
$h_1$-periodic $\rho$-Bockstein spectral sequence
are completely known \cite{guillou-isaksen-eta-R}.
For each $h_1$-periodic element $x$, this 
determines $d_r(h_1^k x)$ for large
values of $k$. However, it is possible that the elements $h_1^k x$ support
shorter differentials for small values of $k$.  By inspection, no such shorter
differentials occur.
\end{proof}

\begin{remark}
The phenomenon considered at the end of the
proof of Proposition \ref{prop:Bock-h1-local} turns out not
to occur through coweight $13$.  However, it does
occur in higher coweights.
\end{remark}

The following examples are representative arguments for establishing
$\rho$-Bockstein differentials.
In many situations, more than one argument leads to the same result.

\begin{ex}
Table \ref{tab:s+f-w=6} summarizes the analysis of Bockstein
differentials in degrees $(s,f,w)$ satisfying $s+f-w = 6$.
In these degrees, the $E_1$-page consists of $\rho$ multiples
of twenty elements.  The first part of Table \ref{tab:s+f-w=6}
lists the two elements that are $\rho$-periodic, as in 
Proposition \ref{prop:rho-local}.  They correspond to the classical
elements $h_0^6$ and $h_0^2 h_2$.

The second section of Table \ref{tab:s+f-w=6} lists some 
differentials that are easily deduced from
Proposition \ref{prop:bock-tau^k} and the Leibniz rule.

At this point, only the elements $\tau^4 h_1^2$ and $c_0$
remain unaccounted.  The third section of Table \ref{tab:s+f-w=6}
gives the only possibility.

\begin{longtable}{lllll}
\caption{Bockstein differentials for $s+f-w = 6$} \\
\toprule 
coweight & $(s,f,w)$ & $x$ & $d_r$ & $d_r(x)$ \\
\midrule \endfirsthead
\caption[]{Bockstein differentials for $s+f-w = 6$} \\
\toprule 
coweight & $(s,f,w)$ & $x$ & $d_r$ & $d_r(x)$ \\
\midrule \endhead
\bottomrule \endfoot
\label{tab:s+f-w=6}
$0$ & $(6, 6, 6)$ & $h_1^6$ \\
$3$ & $(9, 3, 6)$ & $h_1^2 h_3$ \\
\hline
$6$ & $(0,0,-6)$ & $\tau^6$ & $d_2$ & $\tau^5 h_1$ \\
$5$ & $(0,1,-5)$ & $\tau^5 h_0$ & $d_1$ & $\tau^4 h_0^2$ \\
$3$ & $(0,1,-3)$ & $\tau^3 h_0^3$ & $d_1$ & $\tau^2 h_0^4$ \\
$1$ & $(0,1,-1)$ & $\tau h_0^5$ & $d_1$ & $h_0^6$ \\
$4$ & $(3, 2, -1)$ & $\tau^3 h_0 h_2$ & $d_1$ & $\tau^3 h_1^3$ \\
$5$ & $(7, 1, 2)$ & $\tau^2 h_3$ & $d_2$ & $\tau h_1 h_3$ \\
$4$ & $(7, 2, 3)$ & $\tau h_0 h_3$ & $d_1$ & $h_0^2 h_3$ \\
$5$ & $(3, 1, -2)$ & $\tau^4 h_2$ & $d_4$ & $\tau^2 h_2^2$ \\
\hline
$4$ & $(2, 2, -2)$ & $\tau^4 h_1^2$ & $d_7$ & $c_0$ \\
\end{longtable}
\end{ex}

\begin{ex}
In some situations, a more careful analysis of multiplicative
structure establishes a differential.
For example, $d_1(f_0)$
cannot equal $\rho h_1 e_0$ because $h_1 f_0 = 0$
but $\rho h_1^2 e_0$ is not zero.

For a slightly more complicated example,
consider the relation $h_0 \cdot \tau g = \tau \cdot h_0 g$.
This implies that
\[
h_0 \cdot d_1(\tau g) = d_1 (\tau) \cdot h_0 g = 
\rho h_0^2 g,
\]
so $d_1(\tau g)$ must equal $\rho h_0 g$.
\end{ex}

\begin{ex}
\label{ex:Bockstein-multiplicative}
Sometimes, the multiplicative structure and an already known differential
imply that a certain element is killed by $\rho^k$.
Then that element must be killed by a differential $d_r$ with $r \leq k$.
For example,
the element $\tau^4 h_1^2 h_3 = (\tau^2 h_2)^2 h_2$ 
is a permanent cycle because it is a product of permanent cycles.
There are two possible differentials that could hit a $\rho$-multiple of
it:
$d_4(\tau^6 h_2^2)$ or $d_8(\tau^8 h_1^2)$.
Note that $\tau^4 h_1^2 h_3$
is killed by
$\rho^4$ because of the differential
$d_4(\tau^4) = \rho^4 \tau^2 h_2$.
Therefore, $\rho^4 \tau^4 h_1^2 h_3$ must be hit by a $d_r$ differential
with $r \leq 4$.
The only possibility is that $d_4(\tau^6 h_2^2 ) = \rho^4 \tau^2 h_1^2 h_3$.

This differential can be obtained another way using the Leibniz
rule, the multiplicative relation $\tau^6 h_2^2 =
\tau^4 \cdot \tau^2 h_2 \cdot h_2$, and the differential
$d_4(\tau^4) = \rho^4 \tau^2 h_2$. 
\end{ex}

\begin{ex}
Sometimes one must look ahead
to larger values of $s+f-w$ in order to use
multiplicative relations to rule out differentials. 
For example, in order to
show that $d_4(i) = \rho^4 h_1 c_0 e_0$ 
(in degrees satisfying $s+f-w=18$), we first use other techniques
to rule out possible differentials until it suffices to eliminate the
possibility that $d_{11}(\tau^4 P c_0)$ might equal 
$\rho^{11} h_1 c_0 e_0$. But this would imply that
$d_{11}(\tau^4 P h_1 c_0)$ equals $h_1^2 c_0 e_0$ 
(in degrees satisfying $s+f-w=19$), and
this contradicts the $h_1$-periodic
differential $d_3(P e_0) = \rho^3 h_1^2 c_0 e_0$ from
Table \ref{tab:Bock-h1-local}.
\end{ex}

\begin{ex}
The Leibniz rule implies that certain elements survive at least
to a certain page of the spectral sequence.
For example,
the element $\tau^6 h_3^2$ cannot be hit by a differential,
so it must support a differential.
There are two possibilities:
$d_4(\tau^6 h_3^2)$ might equal $\rho^4 \tau^4 h_1^2 h_4$,
or $d_6(\tau^6 h_3^2)$ might equal $\rho^6 \tau^3 c_1$.
The Leibniz rule and the relation $\tau^6 h_3^2 = \tau^4 \cdot \tau^2 h_3^2$
imply that 
\[
d_4 (\tau^6 h_3^2) = d_4(\tau^4) \cdot \tau^2 h_3^2 =
\rho^4 \tau^2 h_2 \cdot \tau^2 h_3^2 = 0.
\]
Therefore, $d_6(\tau^6 h_3^2)$ must equal $\rho^6 \tau^3 c_1$.
\end{ex}

\begin{ex}
The multiplicative structure implies that certain elements do not
support any differentials because they are the product of elements
that do not support any differentials.
\end{ex}

Extending Example \ref{ex:Bockstein-multiplicative},
sometimes the Massey product structure of $\Ext_\R$ implies that
some element $\rho^k x$ must be zero.  Then
$\rho^k x$ must be the target of a Bockstein $d_r$ differential
for $r \leq k$.  Through coweight 12, we apply this method only once
in the following Lemma \ref{lem:d2-t^2g}.  However, we anticipate
that this approach will become more and more important in higher
coweights.  Massey products in $\Ext_\R$ are discussed below in
Section \ref{section:bockstein-extn} and Table \ref{tab:Massey}.

\begin{lemma}
\label{lem:d2-t^2g}
$d_2(\tau^2 g) = \rho^2 h_2 f_0$.
\end{lemma}

\begin{proof}
Table \ref{tab:Massey} shows that
$h_2 f_0$ equals the Massey product $\an{\tau h_1, h_1^4, h_4}$
in $\Ext_\R$.
Shuffle to obtain
\[
\rho^2 \an{\tau h_1, h_1^4, h_4} =
\an{\rho^2, \tau h_1, h_1^4} h_4,
\]
which equals zero because the last bracket is zero.
Therefore, $\rho^2 h_2 f_0$ is hit by a $d_1$ or $d_2$ differential,
and the only possibility is that $d_2(\tau^2 g) = \rho^2 h_2 f_0$.
\end{proof}

Theorem \ref{thm:Bock} summarizes the results of the
analysis of $\rho$-Bockstein differentials.

\begin{thm}
\label{thm:Bock}
Table \ref{tab:Bock} lists some values of the $\rho$-Bockstein
$d_r$ differentials on multiplicative generators of the
$E_r$-page.  Through coweight $13$,
the $d_r$ differential vanishes on all other multiplicative
generators of the $E_r$-page.
\end{thm}

For legibility, we have not included powers of $\rho$ in the
values of the Bockstein differentials in Table \ref{tab:Bock}.
For example, the first row of the table is to be interpreted as
$d_1(\tau) = \rho h_0$.

\section{Hidden extensions in the $\rho$-Bockstein spectral sequence}
\label{section:bockstein-extn}

Section \ref{sctn:Bock-diff} explains how to obtain the
$E_\infty$-page of the $\rho$-Bockstein spectral sequence
through coweight $12$.
As usual, this $E_\infty$-page is an associated graded
object of $\Ext_\R$.

We abuse notation and use the same name for generators
of the $\rho$-Bockstein $E_\infty$-page and
elements of $\Ext_\R$ that they represent.
A generator of the $\rho$-Bockstein $E_\infty$-page can
represent more than one element in $\Ext_\R$, where the indeterminacy is 
parametrized by
elements of the $E_\infty$-page in higher filtration.
For example, the element $\tau^2 h_2$ of the $E_\infty$-page
represents two elements of $\Ext_\R$ whose difference
is $\rho^4 h_3$.

We adopt the following convention in selecting generators
in $\Ext_\R$.  We always choose an element of $\Ext_\R$ that is 
annihilated by the same power of $\rho$ as its representative
in the $E_\infty$-page.  
For example, $\tau^2 h_2$ is annihilated by $\rho^4$ in the
$E_\infty$-page. 
Therefore, we write $\tau^2 h_2$ for the
(unique) element of $\Ext_\R$ that is annihilated by $\rho^4$.
(The other possible choice is $\rho$-periodic.)

This convention concerning annihilation by powers of $\rho$
eliminates much of the ambiguity in passing from the $E_\infty$-page
to $\Ext_\R$.  In some cases, our convention does not eliminate all
ambiguities.  However, the remaining ambiguities make little
practical difference.

In order to recover the full structure of 
$\Ext_\R$ from the $\rho$-Bockstein
$E_\infty$-page, we must determine hidden multiplicative extensions.
We adopt the precise definition of a hidden extension given in
\cite{Isaksen14c}*{Section 4.1.1}.
In this section, we will analyze all hidden extensions
by $h_0$ and $h_1$ through coweight $12$.

The $\rho$-Bockstein spectral sequence has numerous hidden extensions
by other elements.  There are so many examples that it is not practical
to enumerate them exhaustively.  In practice, these other hidden
extensions are occasionally useful, and we treat them on an
ad hoc basis as necessary.

\begin{defn}
\label{defn:decomposable-hidden}
A hidden $a$ extension from $x$ to $y$ is decomposable if there
exists a hidden $a$ extension from $u$ to $v$, and there
exists $z$ such that $x = z u$ and $y = z v$ in the $E_\infty$-page.
\end{defn}

\begin{ex}
\label{ex:decomposable-hidden}
There is a hidden $h_0$ extension from $\tau h_1$ to $\rho \tau h_1^2$.
Multiplication by $\tau h_1$ gives the decomposable
hidden $h_0$ extension 
from $\tau^2 h_1^2$ to $\rho \tau^2 h_1^3$.
\end{ex}

Definition \ref{defn:decomposable-hidden} allows us to focus
only on the hidden extensions that are most significant.
In practice, decomposable hidden extensions are easy to understand,
once the indecomposable hidden extensions have been studied.

\begin{remark}
\label{rem:rho-extn}
The structure of the $\rho$-Bockstein spectral sequence 
guarantees that there are no hidden extensions by $\rho$.
For degree reasons, 
if there is a possible hidden $\rho$ extension from
$x$ to $y$, then in fact $y$ is a multiple of $\rho$.
According to the definition of a hidden extension
\cite{Isaksen14c}*{Section 4.1.1}, this means that $y$ cannot be the target
of a hidden $\rho$ extension.
\end{remark}

\subsection{Massey products}

Our main tool for establishing hidden extensions 
is the May Convergence Theorem 
\cite{May69}*{Theorem 4.1}, restated here for convenience.

\begin{thm}[May Convergence Theorem]\label{thm:may-convergence}
Let $\alpha_0$, $\alpha_1$, and $\alpha_2$ be elements of $\Ext_\R$ such that
the Massey product $\an{\alpha_0,\alpha_1,\alpha_2}$ is defined. For each $i$,
let $a_i$ be a permanent cycle in the Bockstein $E_r$-page that detects
$\alpha_i$. Suppose further that:
\begin{enumerate} 
\item there exist elements $a_{01}$ and $a_{12}$ in the 
Bockstein $E_r$-page
such that $d_r(a_{01})$ equals $a_0a_1$ and $d_r(a_{12})$ equals $a_1a_2$;
\item 
\label{condition:crossing}
if either $a_{01}$ or $a_{12}$ has degree $(s,f,w)$ and $\rho$-Bockstein
degree $m$, 
and $x$ is an element in degree $(s,f,w)$ and $\rho$-Bockstein
degree $m'$ such that $m' \leq m$, then $d_t(x) = 0$ for all
$t$ such that $m'+t > (m -m') + r$.
\end{enumerate}
Then $a_0a_{12}+a_{01}a_2$ is a permanent cycle in the
$\rho$-Bockstein spectral sequence, and it detects an element of
$\an{\alpha_0,\alpha_1,\alpha_2}$ in $\Ext_\R$.
\end{thm}

We will often use Theorem \ref{thm:may-convergence} in the situation
when $a_{01}$ has $\rho$-Bockstein degree $0$ and 
$a_{12}$ has negative $\rho$-Bockstein degree.  
Since the $\rho$-Bockstein spectral sequence is zero in
negative $\rho$-Bockstein degrees,
condition (\ref{condition:crossing}) of Theorem \ref{thm:may-convergence}
simplifies to the condition that no element in the
same degree as $a_{01}$ with $\rho$-Bockstein degree $0$
supports a longer differential.

\begin{prop}\label{prop:Massey}
Table \ref{tab:Massey} lists some Massey products in $\Ext_\R$.
\end{prop}

\begin{proof}
Most of these Massey products are straightforward
applications of the May Convergence Theorem \ref{thm:may-convergence}.
In those cases,
the sixth column of Table \ref{tab:Massey} gives the
$\rho$-Bockstein differential that is relevant for computing
the Massey product.

In some cases, the Massey products follow by comparison to the $\C$-motivic
case.  This is denoted by the word ``$\C$-motivic" in the sixth column of Table
\ref{tab:Massey}. However, this only determines the Massey product up to
multiples of $\rho$. 
These ambiguities can typically be eliminated by the multiplicative
structure.  In particular, 
if the Massey product $\an{x,y,z}$ is defined and $\rho^a x$
and $\rho^b z$ are both zero, then
\[
\rho^{a+b} \an{x,y,z} =
\rho^b \an{\rho^a, x, y} z = 0.
\]

The indeterminacies can be computed by inspection.
\end{proof}

Table \ref{tab:Massey} is not meant to be an exhaustive list of Massey
products.
It merely provides an assortment of Massey products that are 
needed for various specific computations throughout the manuscript.

\subsection{Hidden $h_0$ extensions}

\begin{prop}\label{prop:h0-extn}
Table \ref{tab:Bock-h0-extn} lists all
indecomposable hidden $h_0$ extensions in the
$\rho$-Bockstein spectral sequence, through coweight $12$.
\end{prop}

\begin{proof}
All of the hidden $h_0$ extensions in Table \ref{tab:Bock-h0-extn}
are proved using a single technique, 
which was introduced in the proof of
\cite{DI17}*{Lemma 6.2}. 
To illustrate this technique, we will show that there is a hidden
$h_0$ extension from $\tau^2 h_1 c_0$ to $\rho^2 P h_2$.

First we show that the product $h_0 \cdot \tau^2 h_1 c_0$ is nonzero in $\Ext_\R$. If not,
then the Massey product $\an{\rho, h_0, \tau^2 h_1 c_0}$ would be defined in
$\Ext_\R$. 
The May Convergence Theorem \ref{thm:may-convergence},
together with the $\rho$-Bockstein differential
$d_1(\tau) = \rho h_0$, would then imply that $\tau^3 h_1 c_0$
is a permanent cycle.
But this contradicts the $\rho$-Bockstein differential
$d_3(\tau^3 h_1 c_0) = \rho^3 P h_2$.

This shows that there must be a hidden $h_0$ extension
on $\tau^2 h_1 c_0$.  The target of this hidden extension
can only be $\rho^2 P h_2$ or $\tau P h_1$.
But the target must have higher $\rho$-Bockstein
filtration than the source,
which rules out $\tau P h_1$. 

In some cases, one needs to use multiplicative
relations to rule out possible hidden $h_0$ extensions.
For example, the target of a hidden $h_0$ extension cannot support 
a $\rho$ multiplication, since $\rho h_0 = 0$ in $\Ext_\R$.

We must also show that many elements do not support hidden
$h_0$ extensions.
In all cases through coweight $12$, the non-existence
follows from simple multiplicative relations.
For example, if $x$ is already
known to not support an $h_0$ extension, then the product $xy$
cannot support an $h_0$ extension.
Similarly, if $h_1 y$ or $\rho y$ is non-zero, 
then $y$ cannot be the target
of a hidden extension because of the relations $h_0 h_1 = 0$
and $\rho h_0 = 0$ in $\Ext_\R$.
\end{proof}

\subsection{Hidden $h_1$ extensions}

\begin{prop}
\label{prop:h1-extn}
Table \ref{tab:Bock-h1-extn} lists all indecomposable hidden $h_1$
extensions in the $\rho$-Bockstein
spectral sequence, through coweight $12$.
\end{prop}

\begin{proof}
Many of the extensions are established using the map
\[
\xymatrix@1{
\Ext_\C \ar[r]^{p} & \Ext_\R
}
\]
of Remark \ref{rem:Ext-R-C-exact}.
To illustrate this technique,
we will show that there is a hidden $h_1$ extension from
$\tau^2 h_1 c_0$ to $\rho P h_2$.
The relation $h_1 \cdot \tau^3 c_0 = \tau^3 h_1 c_0$
in $\Ext_\C$ implies that
$h_1 \cdot p(\tau^3 c_0) = p(\tau^3 h_1 c_0)$.
Observe that
$p(\tau^3 c_0) = \rho \tau h_1 \cdot \tau c_0$ and
$p(\tau^3 h_1 c_0) = \rho^2 P h_2$.
This shows that there is a hidden $h_1$ extension from
$\rho \tau^2 h_1 c_0$ to $\rho^2 P h_2$, and it follows that there is
also a hidden $h_1$ extension from
$\tau^2 h_1 c_0$ to $\rho P h_2$.

Several more difficult cases are established in the following lemmas.

We must also show that many elements do not support hidden
$h_1$ extensions.
In most cases through coweight 12, the non-existence
follows from simple multiplicative relations.
For example, if $x$ is already
known to not support an $h_1$ extension, then the product $xy$
cannot support an $h_1$ extension.
Similarly, if $h_0 y$ is non-zero, 
then $y$ cannot be the target
of a hidden $h_1$ extension because of the relation $h_0 h_1 = 0$
in $\Ext_\R$. 

Additionally, the map $p: \Ext_\C \map \Ext_\R$ can be used to
detect the absence of some $h_1$ extensions.
\end{proof}

\begin{remark}
\label{rem:h1-extn-low}
The first three extensions in Table \ref{tab:Bock-h1-extn}
were established in \cite{DI17}.
\end{remark}

\begin{lemma}\label{lem:h1 t^3 h1^2 h3}
There is a hidden $h_1$ extension from $\tau^3 h_2^3$ to
$\rho^4 d_0$.
\end{lemma}

\begin{proof}
The element $\tau^3 h_2^3$ of the $\rho$-Bockstein
$E_\infty$-page detects the element
$\tau^2 h_2 \cdot \tau h_2^2$ in $\Ext_\R$.
Table \ref{tab:Bock-h1-extn} shows that
$h_1 \cdot \tau h_2^2 = \rho c_0$, and
$h_1^2 \cdot \tau^2 h_2 = \rho^3 c_0$.
Therefore,
\[
h_1^3 \cdot \tau^2 h_2 \cdot \tau h_2^2 =
\rho^3 c_0 \cdot \rho c_0 = \rho^4 h_1^2 d_0.
\]
It follows that $h_1 \cdot \tau^2 h_2 \cdot \tau h_2^2$ equals
$\rho^4 d_0$.
\end{proof}

\begin{lemma}\label{lem:t^2 f0 * h1}
There is a hidden $h_1$ extension from $\tau^2 f_0$ to $\rho^2 \tau^2 h_1 g$.
\end{lemma}

\begin{proof}
Table \ref{tab:Massey} shows that $\tau^2 f_0$ belongs to the Massey
product $\an{\tau^2 h_2, h_3, h_0^2 h_3}$.
Table \ref{tab:Bock-h1-extn} shows that there is a hidden
$h_1$ extension from $\tau^2 h_2$ to $\rho^2 \tau h_2^2$.  Therefore, we have
\[
h_1 \an{\tau^2 h_2, h_3, h_0^2 h_3} =
\an{\rho^2 \tau h_2^2, h_3, h_0^2 h_3} =
\rho^2 \an{\tau h_2^2, h_3, h_0^2 h_3},
\]
where the equalities follow from inspection of indeterminacies.
Table \ref{tab:Massey} shows that the element $\tau^2 h_1 g$ of 
the Bockstein $E_\infty$-page detects both elements of the
Massey product $\an{\tau h_2^2, h_3, h_0^2 h_3}$,
so $\rho^2 \tau^2 h_1 g$ is the target of the hidden $h_1$ extension.
\end{proof}

\begin{lemma}\label{lem:t^8 h1 c0 * h1}
\mbox{}
\begin{enumerate}
\item
There is a hidden $h_1$ extension from $\tau^8 h_1 c_0$ to $\rho \tau^6 P h_2$.
\item
There is a hidden $h_1$ extension from $\tau^6 P h_2$ to $\rho^2 \tau^5 h_0^2
d_0$.
\item
There is a hidden $h_1$ extension from $\tau^4 P h_1 c_0$ to 
$\rho \tau^2 P^2 h_2$.
\item
There is a hidden $h_1$ extension from $\tau^2 P^2 h_2$ to 
$\rho^2 \tau P h_0^2 d_0$.
\end{enumerate}
\end{lemma}

\begin{proof}
We will show that $h_1^3 \cdot \tau^8 c_0$ equals $\rho^3 \tau^5 h_0^2 d_0$.
This will establish the first two extensions simultaneously.

Table \ref{tab:Massey} shows that $h_1 \cdot \tau^8 c_0$ equals the 
Massey product $\an{\tau h_1 \cdot \tau^5 c_0, \tau h_1, \rho^2}$.
By inspection of indeterminacies,
\[
h_1^2 \an{\tau h_1 \cdot \tau^5 c_0, \tau h_1, \rho^2} =
h_1 \an{h_1 \cdot \tau h_1 \cdot \tau^5 c_0, \tau h_1, \rho^2}.
\]
This expression equals
$h_1 \an{\rho \tau^4 P h_2, \tau h_1, \rho^2}$, since 
Table \ref{tab:Bock-h1-extn} shows that there is a hidden
$h_1$ extension from $\tau^6 h_1 c_0$ to $\rho \tau^4 P h_2$.
By inspection of indeterminacies again, this also equals
$\rho h_1 \an{\tau^4 P h_2, \tau h_1, \rho^2}$.

Now shuffle to obtain
\[
\rho h_1 \an{\tau^4 P h_2, \tau h_1, \rho^2} = 
\rho^3 \an{h_1, \tau^4 P h_2, \tau h_1}.
\]
Finally, Table \ref{tab:Massey} shows that
$\an{h_1, \tau^4 P h_2, \tau h_1}$ equals
$\tau^5 h_0^2 d_0$.
This establishes the first two extensions.

The argument for the last two extensions is essentially identical.
The Massey product $\an{\tau h_1 \cdot \tau P c_0, \tau h_1, \rho^2}$
equals $h_1 \cdot \tau^4 P c_0$.
We have
\[
h_1^2 \an{\tau h_1 \cdot \tau P c_0, \tau h_1, \rho^2} =
h_1 \an{h_1 \cdot \tau h_1 \cdot \tau P c_0, \tau h_1, \rho^2},
\]
which equals
\[
h_1 \an{\rho P^2 h_2, \tau h_1, \rho^2} =
\rho h_1 \an{P^2 h_2, \tau h_1, \rho^2}.
\]
Finally, shuffle to obtain
\[
\rho h_1 \an{P^2 h_2, \tau h_1, \rho^2} = 
\rho^3 \an{h_1, P^2 h_2, \tau h_1} = \rho^3 \tau P h_0^2 d_0.
\]
\end{proof}

\begin{lemma}\label{lem:h1 * t^3 c1}
There is a hidden $h_1$-extension from $\tau^3 c_1$ to $\rho^2 \tau^2
h_2 c_1$.
\end{lemma}

\begin{proof}
Table \ref{tab:Massey} shows that
$\tau^3 c_1$ is contained in the Massey product
$\an{\rho^2, \tau h_1, \tau c_1}$.  Shuffle to obtain
\[
\an{\rho^2, \tau h_1, \tau c_1} h_1 =
\rho^2 \an{\tau h_1, \tau c_1, h_1}.
\]
Table \ref{tab:Massey} shows that
the element $\tau^2 h_2 c_1$ of the Bockstein $E_\infty$-page
detects both elements of $\an{\tau h_1, \tau c_1, h_1}$,
so $\rho^2 \tau^2 h_2 c_1$ is the target of the hidden $h_1$ extension.
\end{proof}

\begin{lemma}\label{lem:t^3 h2^2 e0}
~\begin{enumerate} 
\item There is a hidden $h_1$ extension from $\tau^3 h_2^2 e_0$ to $\rho^2 j$.
\item There is a hidden $h_1$ extension from $j$ to $\rho d_0^2$.
\end{enumerate}
\end{lemma}

\begin{proof}
Table \ref{tab:Bock-h1-extn} shows that
$h_1 \cdot \tau h_2^2 = \rho c_0$, and
$h_1^3 \cdot \tau^2 e_0 = h_1 \cdot \rho \tau h_2^2 \cdot d_0 =
\rho^2 c_0 d_0$.
Therefore,
\[
h_1^4 \cdot \tau h_2^2 \cdot \tau^2 e_0 =
\rho^3 c_0^2 d_0 = \rho^3 h_1^2 d_0^2.
\]
Both hidden extensions are immediate consequences.
\end{proof}

\subsection{Miscellaneous relations}

We briefly consider a few other types of hidden extensions.

In the Bockstein $E_\infty$-page, we have the relation
$h_1^2 \cdot \tau^4 h_3 + (\tau^2 h_2)^2 h_2 = 0$.
However, in $\Ext_\R$,
it is possible that the sum
$h_1^2 \cdot \tau^4 h_3 + (\tau^2 h_2)^2 h_2$ equals
a non-zero element that is detected 
in higher $\rho$-Bockstein filtration.
Lemma \ref{lem:compound} 
demonstrates that this does in fact occur.
It provides one additional piece of information
about the multiplicative structure of $\Ext_\R$.

\begin{lemma}
\label{lem:compound}
In $\Ext_\R$ we have the relation 
\[
h_1^2 \cdot \tau^4 h_3 + (\tau^2 h_2)^2 h_2 = \rho^5 \tau h_0 h_3^2.
\]
\end{lemma}

\begin{proof}
This follows by comparison along the map
$p: \Ext_\C \map \Ext_R$ of Remark \ref{rem:Ext-R-C-exact}.
The relation $h_1 \cdot \tau^8 h_1 = \tau^8 h_1^2$ in
$\Ext_\C$ implies that $h_1 \cdot p(\tau^8 h_1) = p(\tau^8 h_1^2)$
in $\Ext_\R$.
Observe that
$p(\tau^8 h_1) = \rho^7 \tau^4 h_1 h_3$ and
$p(\tau^8 h_1^2) = \rho^{12} \tau h_0 h_3^2$.
This shows that there is a hidden $h_1$ extension from
$\rho^7 \tau^4 h_1 h_3$ to $\rho^{12} \tau h_0 h_3^2$,
which implies the desired relation.
\end{proof}

\begin{lemma}
\label{lem:t^2h2 * c0}
There is a hidden $\tau^2 h_2$ extension from $c_0$
to $\rho^3 d_0$.
\end{lemma}

\begin{proof}
Table \ref{tab:Bock-h1-extn} shows that there are hidden $h_1$
extensions from $\tau h_2^2$ to $\rho c_0$, and from
$\tau^3 h_2^2$ to $\rho^4 d_0$.
Therefore,
\[
\tau^2 h_2 \cdot \rho c_0 =
\tau^2 h_2 \cdot h_1 \cdot \tau h_2^2 = \rho^4 d_0.
\]
\end{proof}

\begin{lemma}
\label{lem: h2 * h2 f0}
There is a hidden $h_2$ extension from $h_2 f_0$ to 
$\rho h_1^2 h_4 c_0$.
\end{lemma}

\begin{proof}
We use the map $p: \Ext_\C \map \Ext_\R$
of Remark \ref{rem:Ext-R-C-exact}.
The relation $h_2 \cdot \tau^2 g = \tau^2 h_2 g$
in $\Ext_\C$ implies that
$h_2 \cdot p(\tau^2 g) = p(\tau^2 h_2 g)$.
Observe that $p(\tau^2 g) = \rho h_2 f_0$, and
$p(\tau^2 h_2 g) = \rho^2 h_1^2 h_4 c_0$.

Therefore, there is a hidden $h_2$ extension from
$\rho h_2 f_0$ to $\rho^2 h_1^2 h_4 c_0$, and also
a hidden $h_2$ extension from $h_2 f_0$ to $\rho h_1^2 h_4 c_0$.
\end{proof}

\section{Adams differentials}
\label{sctn:Adams-diff}

Sections \ref{sctn:Bock-diff} and \ref{section:bockstein-extn}
describe how to compute $\Ext_\R$, which serves as the
$E_2$-page of the $\R$-motivic Adams spectral sequence.
We now proceed to analyze Adams differentials.
We remind the reader of the notation for stable homotopy elements
discussed in Section \ref{subsctn:pi-notation} and
Table \ref{tab:pi-notation}.

Recall from Section \ref{sctn:compare-R-C} that extension of scalars
induces a map from the $\R$-motivic Adams spectral sequence
to the $\C$-motivic Adams spectral sequence.
We will frequently use these comparison
functors to deduce information about the $\R$-motivic Adams spectral
sequence from already known information about the $\C$-motivic
and classical Adams spectral sequences.
See \cite{Isaksen14c} for an extensive summary of
computational information about the
$\C$-motivic and classical Adams spectral sequences.

\subsection{Toda brackets}

The Moss Convergence Theorem \ref{thm:Moss} is a key tool for determining Toda brackets \cite{Moss70} \cite{Isaksen14c}*{Section 3.1}.
We restate a version of the theorem here for convenience.

\begin{thm}[Moss Convergence Theorem]
\label{thm:Moss}
Let $\alpha_0$, $\alpha_1$, and $\alpha_2$ be elements of the 
$\R$-motivic stable homotopy groups such that
the Toda bracket $\an{\alpha_0, \alpha_1, \alpha_2}$
is defined.
Let $a_i$ be a permanent cycle on the Adams $E_r$-page that detects
$\alpha_i$ for each $i$.
Suppose further that:
\begin{enumerate}
\item
the Massey product $\an{a_0, a_1, a_2}_{E_{r}}$ is defined
(in $\Ext_\R$ when $r=2$, or 
using the Adams $d_{r-1}$ differential when $r \geq 3$).
\item
if $(s,f,w)$ is the degree of either $a_0 a_1$ or $a_1 a_2$;
$f' < f - r + 1$; $f'' > f$; and $t = f''-f'$; then every 
Adams differential 
$d_{t} : E_{t}^{s+1,f',w} \map E_{t}^{s,f'',w}$
is zero.
\end{enumerate}
Then
$\an{a_0, a_1, a_2}_{E_{r}}$ contains a permanent cycle
that detects an element of the Toda bracket
$\an{\alpha_0, \alpha_1, \alpha_2}$.
\end{thm}

\begin{thm}\label{thm:Toda}
Table \ref{tab:Toda} lists some Toda brackets in $\pi_{*,*}$.
\end{thm}

\begin{proof}
Most of these Toda brackets are straightforward
applications of the Moss Convergence Theorem \ref{thm:Moss}.
When a Massey product appears in the fifth column of Table
\ref{tab:Toda}, the Toda bracket follows from
the Moss Convergence Theorem \ref{thm:Moss} with $r = 2$.
When an Adams differential appears in the fifth column of
Table \ref{tab:Toda}, the Toda bracket follows from
the Moss Convergence Theorem \ref{thm:Moss} with $r > 2$, and
the given Adams differential is relevant for computing the 
Toda bracket.

In some cases, the Toda brackets follow by comparison
along the extension of scalars functor
to the $\C$-motivic case.  This is denoted by the word
``$\C$-motivic" in the fifth column of Table \ref{tab:Toda}.

One slightly different case is handled below in 
Lemma \ref{lem:<rho^2, tau eta, nu_4>}. 
\end{proof}

Table \ref{tab:Toda} is not meant to be exhaustive in any sense.
It merely provides the Toda brackets that are 
needed for various specific computations.
Beware that these brackets have non-trivial indeterminacies,
although we have not specified the indeterminacies because they
are not generally relevant to our specific needs.

Beware that some of the Toda brackets in Table \ref{tab:Toda} 
require knowledge of Adams differentials that are established below
in Section \ref{subsctn:Adams-d2}.

\begin{lemma}
\label{lem:<rho^2, tau eta, nu_4>}
The Toda bracket $\an{\rho^2, \tau \eta, \nu_4}$
is detected by $\tau^2 h_2 \cdot h_4$.
\end{lemma}

\begin{proof}
Table \ref{tab:Massey} shows that 
$\tau^2 h_2$ is contained in the Massey product 
$\an{\rho^2, \tau h_1, h_2}$.
By inspection of indeterminacies,
\[
\tau^2 h_2 \cdot h_4 = 
\an{\rho^2, \tau h_1, h_2} h_4 =
\an{\rho^2, \tau h_1, h_2 h_4}.
\]
The Moss Convergence Theorem \ref{thm:Moss} implies that
$\tau^2 h_2 \cdot h_4$ detects the corresponding Toda bracket.
\end{proof}

\subsection{Adams $d_2$ differentials}
\label{subsctn:Adams-d2}

We now proceed to analyze Adams differentials.

\begin{thm}
\label{thm:adams-d2}
Table \ref{tab:adams-d2} lists some values of the
$\R$-motivic Adams $d_2$ differential.
Through coweight $12$,
the $d_2$ differential is zero on all other multiplicative
generators of the $\R$-motivic Adams $E_2$-page.
\end{thm}

\begin{proof}
The multiplicative structure rules out many possible differentials.
For example, $d_2(\tau^5 h_1)$ cannot equal $\tau^4 h_0 \cdot h_0^2$
because $h_0^2 \cdot \tau^5 h_1 = 0$, while $\tau^4 h_0 \cdot h_0^4$ is non-zero.

Other multiplicative generators are known to be permanent cycles,
because the Moss Convergence Theorem \ref{thm:Moss} shows that
they must survive to detect various Toda brackets.
These instances are shown in Table \ref{tab:perm}.
In one case, the element $h_4 \cdot \tau c_0$ must survive to 
detect the product $\sigma \cdot \tau \eta_4$, by comparison to the
$\C$-motivic stable homotopy groups.

Many non-zero differentials follow by comparison to the $\C$-motivic
or classical Adams spectral sequences.

Several more difficult cases are established in the following
lemmas.
\end{proof}

\begin{remark}
\label{rem:perm-t^4h3}
Table \ref{tab:perm} shows that $\tau^4 h_3$ is a permanent cycle
because it detects the Toda bracket $\an{\rho^4, \tau^2 \nu, \sigma}$.
We give an alternative proof that is geometrically interesting,
following the method of \cite{DI17}*{Lemma 7.3}.

There is a functor from classical homotopy theory
to $\R$-motivic homotopy theory that takes the sphere
$S^p$ to $S^{p,0}$.
Let $\sigma_{\tp}: S^{15,0} \map S^{8,0}$ 
be the image of the classical Hopf map $\sigma: S^{15} \map S^8$
under this functor.

The cohomology of the cofiber of $\sigma_{\tp}$ 
is free on two generators $x$ and $y$ of degrees
$(8,0)$ and $(16,0)$, satisfying
$\Sq^8(x) = \tau^4 y$ and $\Sq^{16}(x) = \rho^8 y$.
The proof of these formulas is essentially identical
to the proof of \cite{DI17}*{Lemma 7.4}.

This shows that $\tau^4 h_3 + \rho^8 h_4$ is a permanent
cycle in the Adams spectral sequence, since 
it detects the stabilization of $\sigma_{\tp}$ in
$\pi_{7,0}$.
Also, $\rho^8 h_4$ is a permanent cycle because there are 
no possible values for differentials.
Therefore, $\tau^4 h_3$ is a permanent cycle.
\end{remark}

\begin{lemma}
\label{lem:d2(th0h3^2)}
$d_2(\tau h_0 h_3^2) = \rho^2 h_1 d_0$.
\end{lemma}

\begin{proof}
Table \ref{tab:adams-d2} shows that $d_2(e_0) = h_1^2 d_0$.
Therefore,
\[
d_2(h_1 \cdot \tau h_0 h_3^2) = d_2(\rho^2 e_0) = \rho^2 h_1^2 d_0.
\]
It follows that $d_2(\tau h_0 h_3^2)$ equals $\rho^2 h_1 d_0$.
\end{proof}

\begin{lemma}\label{lem:d2(f_0)}
$d_2(f_0) = h_0^2 e_0$.
\end{lemma}
\begin{proof}
Comparison to the $\C$-motivic or classical case shows that
$d_2(f_0)$ equals either $h_0^2 e_0$ or $h_0^2 e_0 + \rho^2 h_1^2 e_0$.
But $h_1 \cdot f_0 = 0$ in the $E_2$-page, while
$h_1 (h_0^2 e_0 + \rho^2 h_1^2 e_0)$ is non-zero.
The only possibility is that
$d_2(f_0)$ equals $h_0^2 e_0$.
\end{proof}

\begin{lemma}\label{lem:d2(t^2 f0)}
$d_2(\tau^2 f_0) = h_0^2 \cdot \tau^2 e_0 + \rho^3 \tau h_2^2 \cdot d_0$.\end{lemma}
\begin{proof}
The $\C$-motivic differential $d_2(\tau^2 f_0)=\tau^2 h_0^2 e_0$ implies that
$d_2(\tau^2 f_0)$ equals either $h_0^2 \cdot \tau^2 e_0$ or 
$h_0^2 \cdot \tau^2 e_0 + \rho^3 \tau h_2^2 \cdot d_0$. We rule out the first possibility by noting that
$(h_0^2 + \rho^2 h_1^2)\cdot \tau^2 f_0 = 0$ in $\Ext_\R$ whereas $(h_0^2 +
\rho^2 h_1^2)\cdot \tau^2 h_0^2 e_0 = \rho^6 h_1 c_0 d_0$.
\end{proof}

\begin{lemma}\label{d_2(t^2 h1 g)}
$d_2(\tau^2 h_1 g) = \rho^2 c_0 d_0$.
\end{lemma}
\begin{proof}
Table \ref{tab:Bock-h1-extn} shows that 
$h_1 \cdot \tau^2 h_1 g = \rho \tau h_2^2 \cdot e_0$.
Therefore,
\[
h_1 \cdot d_2(\tau^2 h_1 g) = \rho \tau h_2^2 \cdot d_2(e_0) = 
\rho \tau h_2^2 \cdot h_1^2 d_0,
\]
which equals $\rho^2 h_1 c_0 d_0$ because Table \ref{tab:Bock-h1-extn}
shows that $h_1 \cdot \tau h_2^2 = \rho c_0$.
\end{proof}

\subsection{Higher Adams differentials}

Theorem \ref{thm:adams-d2} completely describes the Adams
$d_2$ differential through coweight $12$.
From this information, one can compute the Adams $E_3$-page in 
a range.  We now proceed to analyze higher differentials.

\begin{thm}
\label{thm:adams-higher}
Table \ref{tab:higher-adams} lists some values of the
$\R$-motivic Adams $d_3$ differential for $r \geq 3$.
Through coweight $12$,
the $d_3$ differential is zero on all other multiplicative
generators of the $\R$-motivic Adams $E_3$-page.
Moreover, through coweight $12$, there are no higher differentials,
and the $\R$-motivic Adams $E_4$-page equals the 
$\R$-motivic Adams $E_\infty$-page.
\end{thm}

\begin{proof}
As in the proof of Theorem \ref{thm:adams-d2}, many
multiplicative generators cannot support differentials because
there are no possible targets.
Comparison to the $\C$-motivic and classical cases 
also determines some differentials.
For example, $d_3(h_1 h_4)$ cannot equal $h_1 d_0$.

Other multiplicative generators are known to be permanent cycles,
because the Moss Convergence Theorem \ref{thm:Moss} shows that
they must survive to detect various Toda brackets.
These instances are shown in Table \ref{tab:perm}.

The multiplicative structure rules out additional cases.
For example $d_3(\rho h_4)$ cannot equal $\rho d_0$
because of the relation $h_1 \cdot \rho h_4 = \rho \cdot h_1 h_4$,
together with the fact that $d_3(h_1 h_4)$ is already known
to be zero.

The harder cases are established
in the following lemmas.
\end{proof}

\begin{lemma}\label{lem:p^6 e0 perm}
$d_3(\rho^6 e_0) = 0$.
\end{lemma}

\begin{proof}
If $d_3(\rho^6 e_0)$ equaled $\rho h_1 \cdot \tau h_1\cdot \tau P h_1$,
then $\rho^7 e_0$ would be a permanent cycle 
that detected an element $\alpha$ of $\pi_{10,3}$, and $\alpha$
could not be divisible by $\rho$.
Therefore, by Corollary \ref{cor:pi-rho},
$\alpha$ would map to a non-zero element $\beta$ in
$\pi^\C_{10,3}$.  Then $\beta$ 
would have to be detected by $\tau^3 P h_1^2$, so 
$\eta \beta$ would also have to be non-zero in $\pi^C_{11,4}$.

But $\eta \alpha$ would be detected by $\rho^7 h_1 e_0$
and would be divisible by $\rho$, so it would map to 
zero in $\pi_{11,4}^\C$.
This contradicts that $\eta \beta$ is non-zero.
\end{proof}

\begin{remark}
\label{rem:kq}
Lemma \ref{lem:p^6 e0 perm} can also be proved using the 
$\R$-motivic spectrum $kq$, which is the very effective
slice cover of the Hermitian $K$-theory spectrum $KQ$
\cite{ARO17}.
The cohomology of $kq$ is isomorphic to $\mathcal{A}//\mathcal{A}(1)$,
where $\mathcal{A}(1)$ is the $\M_2$-subalgebra of the $\R$-motivic
Steenrod algebra that is generated by $\Sq^1$ and $\Sq^2$.

By a change-of-rings isomorphism,
the homotopy of $kq$ is computed by an Adams spectral
sequence whose $E_2$-page is $\Ext_{\mathcal{A}(1)}(\M_2, \M_2)$.
This $E_2$-page was computed in \cite{Hill11}, and also in
\cite{GHIR20}*{Section 6}.

The element $\rho \tau h_1 \cdot \tau P h_1 \cdot h_1$
maps to a non-zero permanent cycle in
\[
\Ext_{A(1)}(\M_2, \M_2),
\]
so it cannot be the target of a differential.
\end{remark}

\begin{lemma}\label{lem:h0 h4} 
$d_3(h_0 h_4) = h_0 d_0 + \rho h_1 d_0$
\end{lemma}
\begin{proof}
The classical differential $d_3(h_0 h_4) = h_0 d_0$ implies that 
in the $\R$-motivic case,
$d_3(h_0 h_4)$ equals either $h_0 d_0$ or
$h_0 d_0 + \rho h_1 d_0$. 

Note that $\tau h_1 \cdot h_0 d_0 = \rho \tau h_1 \cdot h_1 d_0$
is non-zero on the $E_3$-page, but
$\tau h_1 \cdot h_0 h_4 = \rho \tau h_1 \cdot h_1 h_4$ 
is a permanent cycle, as shown in Table \ref{tab:perm}.
Therefore, $d_3(h_0 h_4)$ cannot equal
$h_0 d_0$.
\end{proof}

\begin{lemma}\label{lem:p j}
\mbox{}
\begin{enumerate}
\item
$d_3(\tau h_2^2 \cdot \tau^2 e_0) = \rho \tau P h_1 \cdot d_0$.
\item
$d_3(\rho j) = \tau P h_1 \cdot h_1 d_0$.
\end{enumerate}
\end{lemma}

\begin{proof}
Let $\alpha$ be an element of $\pi_{24,13}$ that is represented by
$\tau P h_1 \cdot h_1 d_0$.
By comparison of Adams spectral sequences, 
extension of scalars must take $\alpha$ to zero
in $\pi^{\C}_{24,13}$.
Moreover, $\tau P h_1 \cdot h_1 d_0$ cannot be the target of a 
hidden $\rho$ extension.
Therefore, by Corollary \ref{cor:pi-rho},
$\tau P h_1 \cdot h_1 d_0$ must be the target of an 
$\R$-motivic Adams differential, and there is only one possible
such differential.
This establishes the second formula.

The first formula follows immediately from the second one, using the 
relation 
$h_1 \cdot \tau h_2^2 \cdot \tau^2 e_0 = \rho c_0 \cdot \tau^2 e_0$.
\end{proof}

\section{Hidden extensions in the Adams spectral sequence}
\label{sctn:Adams-extns}

We have now obtained the Adams $E_\infty$-page through coweight 11.
It remains to determine hidden extensions that are hidden
in the $\R$-motivic Adams spectral sequence.
As in Section \ref{section:bockstein-extn}, we use the
precise definition of a hidden extension given in
\cite{Isaksen14c}*{Section 4.1.1}.
We will analyze all hidden extensions by
$\rho$, $\hsf$, and $\eta$ through coweight $11$.

We begin by analyzing all hidden extensions
by $\rho$. The main tools are Corollaries \ref{cor:pi-rho}
and \ref{cor:LES}.  

\begin{prop}\label{prop:Adams-rho-extn}
Table \ref{tab:Adams-rho-extn} lists all
hidden $\rho$ extensions in the
Adams spectral sequence, through coweight $11$.
\end{prop}

\begin{proof}
The long exact sequence of
Corollary \ref{cor:LES} gives short exact sequences
\[
0 \map (\coker \rho)_{s,w} \map \pi^\C_{s,w} \map
(\ker \rho)_{s,w+1} \map 0.
\]
The rank of $\pi^\C_{s,w}$, which is entirely known
in our range \cite{Isaksen14c} \cite{IWX19},
severely constrains the possible ranks
of $\coker \rho$ and $\ker \rho$.
From these constraints, we can generally deduce the presence and
absence of hidden $\rho$ extensions, and there is typically
only one possibility in each case in the range under consideration.
The only exception is considered below in Lemma \ref{lem:p * t h1 c0 d0}.
\end{proof}

\begin{lemma}\label{lem:p * t h1 c0 d0}
There is a hidden $\rho$ extension from $\tau h_1 c_0 d_0$ to $P h_0 d_0$.
\end{lemma}

\begin{proof}
Table \ref{tab:Adams-eta-extn} shows that there is a hidden
$\eta$ extension from $\rho \tau c_0 \cdot d_0$ to $P h_0 d_0$.
Therefore, there must be a hidden $\rho$ extension from
$h_1 \cdot \tau c_0 \cdot d_0$ to $P h_0 d_0$.
\end{proof}

\begin{thm}
\label{thm:Adams-h}
Table \ref{tab:Adams-h0-extn} lists all hidden $\hsf$ extensions
in the $\R$-motivic Adams spectral sequence,
through coweight $11$.
\end{thm}

\begin{proof}
The long exact sequence of
Corollary \ref{cor:LES} gives short exact sequences
\[
0 \map (\coker \rho)_{s,w} \map \pi^\C_{s,w} \map
(\ker \rho)_{s,w+1} \map 0.
\]
Some of the extensions can be determined via these short exact sequences,
using known $2$ extensions in $\pi^C_{*,*}$.
For example, the element $\rho^6 e_0$ in the
$\R$-motivic Adams $E_\infty$-page lies in $(\coker \rho)_{11,4}$,
and it maps to the element $\tau^2 \zeta_{11}$ in
$\pi^\C_{11,4}$ that is detected by $\tau^2 P h_2$.
But $2 \tau^2 \zeta_{11}$ is non-zero in $\pi^\C_{11,4}$,
so $\hsf \alpha$ must also be non-zero.
It follows that $\rho^6 e_0$ supports a hidden
$\hsf$ extension.

We must also show that many elements do not support hidden
$\hsf$ extensions.
In most of the cases through coweight 11, the non-existence
follows from simple multiplicative relations.
For example, if $x$ is a multiple of $\rho$ or of $h_1$,
then $x$ cannot support a hidden $\hsf$ extension
because of the relations $\rho \hsf = 0$ and $\hsf \eta = 0$.
Similarly, if $h_1 y$ or $\rho y$ is non-zero, 
then $y$ cannot be the target of a hidden $\hsf$ extension.

The following lemmas handle a few additional more complicated cases.
\end{proof}

\begin{lemma}\label{lem:h0 * t h1 g}
There is a hidden $\hsf$ extension from $h_2 f_0$ to $\rho c_0 d_0$.
\end{lemma}

\begin{proof}
Table \ref{tab:Toda} shows that $h_2 f_0$ detects the Toda bracket
$\an{\rho, \{h_2 e_0\}, \eta}$.
Shuffle to obtain
\[
\an{\rho, \{h_2 e_0\}, \eta} \hsf =
\rho \an{ \{h_2 e_0\}, \eta, \hsf}.
\]
Table \ref{tab:Toda} shows that 
$c_0 d_0$ detects the latter bracket.
\end{proof}

\begin{lemma}\label{lem:h0 * t h2^2 h4}
There is no hidden $\hsf$ extension on $\tau h_2^2 \cdot h_4$.
\end{lemma}

\begin{proof}
The only possible target is $\rho \tau c_0 \cdot d_0$.
Table \ref{tab:Adams-eta-extn} shows that
$\rho \tau c_0 \cdot d_0$ supports a hidden $\eta$ extension,
so it cannot be the target of a hidden $\hsf$ extension.
\end{proof}

\begin{lemma}
\label{lem:h0 * p^3 c0 e0}
There is a hidden $\hsf$ extension from
$\tau c_0 \cdot d_0$ to $P h_0 d_0$.
\end{lemma}

\begin{proof}
Let $\alpha$ be an element of $\pi_{8,4}$ that is detected by
$\tau c_0$, so $\tau c_0 \cdot d_0$ detects $\alpha \kappa$.
Table \ref{tab:Adams-rho-extn} shows that there is a hidden
$\rho$ extension from $h_1 \cdot \tau c_0 \cdot d_0$ to
$P h_0 d_0$,
so $P h_0 d_0$ detects $\rho \eta \alpha \kappa$.
But $(\hsf + \rho \eta) \kappa$ is zero,
so $(\hsf + \rho \eta) \alpha \kappa$ must also be zero.
This implies that $\hsf \alpha \kappa$ is also detected
by $P h_0 d_0$.
\end{proof}

\begin{lemma}
\label{lem:h0 * h4 c0}
There is no hidden $\hsf$ extension on $h_4 c_0$.
\end{lemma}

\begin{proof}
By comparison to the $\C$-motivic (or classical) case,
$h_4 c_0$ detects the product $\sigma \eta_4$.
By inspection, $\hsf \eta_4$ is zero in $\pi_{16,9}$.
\end{proof}

\begin{thm}
\label{thm:Adams-eta}
Table \ref{tab:Adams-eta-extn} lists some hidden $\eta$ extensions
in the $\R$-motivic Adams spectral sequence, through coweight $11$. 
\end{thm}

\begin{proof}
The long exact sequence of
Corollary \ref{cor:LES} gives short exact sequences
\[
0 \map (\coker \rho)_{s,w} \map \pi^\C_{s,w} \map
(\ker \rho)_{s,w+1} \map 0.
\]
Many of these extensions can be obtained by comparison to the
$\C$-motivic case, using these short exact sequences,
as in the proof of Theorem \ref{thm:Adams-h}.
For example, the element $\rho \tau h_1 \cdot \tau P c_0$ detects
an element $\alpha$ in $(\ker \rho)_{16,7}$.
The pre-image $\beta$ of $\alpha$ in $\pi^\C_{16,6}$ is detected by
$\tau^3 P c_0$.  There is a $\C$-motivic hidden $\eta$ extension 
from $\tau^3 h_0^3 h_4$ to $\tau^3 P c_0$, so
$\beta$ is divisible by $\eta$.
This implies that $\alpha$ is also divisible by $\eta$, 
and that there is an $\R$-motivic hidden $\eta$ extension from
$\tau^2 h_0 \cdot h_0^3 h_4$ to $\rho \tau h_1 \cdot \tau P c_0$.

We must also show that many elements do not support hidden
$\eta$ extensions.
In all cases through coweight 11, the non-existence
follows from simple multiplicative relations.
For example, if $x$ is a multiple of $h_0$,
then $x$ cannot support a hidden $\eta$ extension
because of the relation $\hsf \eta = 0$.
Similarly, if $h_0 y$ is non-zero, 
then $y$ cannot be the target of a hidden $\eta$ extension.
\end{proof}

\begin{lemma}\label{lem:eta * t^2 h3^2}
There is no hidden $\eta$ extension on $\tau^2 h_3^2$.
\end{lemma}

\begin{proof}
Table \ref{tab:Toda} shows that $\tau^2 h_3^2$ detects the Toda bracket
$\an{\tau^2 \nu, \sigma, \nu}$.  Shuffle to obtain
\begin{align*}
\an{\tau^2 \nu, \sigma, \nu} \eta  & = \tau^2 \nu \an{\sigma,\nu,\eta}.
\end{align*}
The latter bracket is zero.
\end{proof}

\begin{lemma}
\label{lem:eta * t c1}
There is no hidden $\eta$ extension on $\tau c_1$.
\end{lemma}

\begin{proof}
The possible target $\rho h_2 f_0$ is ruled out by the fact that
$\rho h_2 f_0$ supports an $h_2$ extension, as shown in
Lemma \ref{lem: h2 * h2 f0}.
The possible target $\tau h_2^2 \cdot d_0$ is ruled out by comparison
to the $\C$-motivic case.
\end{proof}

\section{Extension of scalars}
\label{sctn:extn-scalars}

We will now study the values of the extension
of scalars map $\pi^\R_{*,*} \map \pi^\C_{*,*}$.
Corollary \ref{cor:pi-rho} tells us exactly which elements
of $\pi^\R_{*,*}$ have non-trivial images in
$\pi^\C_{*,*}$.
This information about extension of scalars is essential
to our approach to the Mahowald invariant described in 
Section \ref{sctn:root}.

For the most part, the extension of scalars map is detected by
the map from the $\R$-motivic Adams $E_\infty$-page to the
$\C$-motivic Adams $E_\infty$-page.
For example, the element $(\tau \eta)^2$ 
of $\pi^\R_{2,0}$ is detected by $\tau h_1^2$
in the $\R$-motivic Adams $E_\infty$-page, so
its image in $\pi^\C_{2,0}$ must be $\tau^2 \eta^2$,
which is detected by $\tau^2 h_1^2$ in the
$\C$-motivic Adams $E_\infty$-page.

However, there are a few values that are hidden by the
Adams spectral sequence.  In other words, there
exist elements $\alpha$ in $\pi^\R_{*,*}$ such that the
Adams filtration of $\alpha$ is strictly less than the
Adams filtration of its image in $\pi^\C_{*,*}$.

\begin{thm}
\label{thm:R-to-C}
Through coweight $11$, 
Table \ref{tab:R-to-C} lists all
hidden values of the 
extension of scalars map $\pi_{*,*}^\R \map \pi_{*,*}^\C$.
\end{thm}

\begin{proof}
We inspect all elements of the $\R$-motivic Adams $E_\infty$-page
that are not targets of $\rho$ extensions.
Most of these elements map non-trivially to the
$\C$-motivic Adams $E_\infty$-page.
For example, $(\tau h_1)^2$ maps to $\tau^2 h_1^2$.

A few elements map to zero in the $\C$-motivic Adams
$E_\infty$-page.  We treat these elements individually.
In some cases, there is only one possible target in
sufficiently high Adams filtration.
The remaining cases are handled by the following lemmas.
\end{proof}

\begin{lemma}
\label{lem:R-to-C-ph4}
Extension of scalars takes 
elements detected by $\rho h_4$ to elements detected by
$\tau h_3^2$.
\end{lemma}

\begin{proof}
Table \ref{tab:Toda} shows that 
$\rho h_4$ detects 
the Toda bracket $\an{\rho, \hsf, \sigma^2}$.
Extension of scalars takes $\an{\rho, \hsf, \sigma^2}$ 
in $\pi_{14,7}^\R$ to
$\an{0, 2, \sigma^2}$ in $\pi_{14,7}^\C$, 
which equals $\{0, \tau \sigma^2\}$.
The only non-zero value is $\tau \sigma^2$, which is detected by
$\tau h_3^2$.
\end{proof}

\begin{lemma}
\label{lem:R-to-C-pf0}
Extension of scalars takes elements detected by $\rho f_0$ to 
elements detected by $\tau h_2 d_0$.
\end{lemma}

\begin{proof}
Table \ref{tab:Toda} shows that $\rho f_0$ detects
the Toda bracket $\an{\rho, \hsf, \nu \kappa}$.
Extension of scalars takes
$\an{\rho, \hsf, \nu \kappa}$ in $\pi_{17,9}^\R$ to
$\an{0, 2, \nu \kappa}$ in $\pi_{17,9}^\C$, which equals
$\{0, \tau \nu \kappa\}$.
The only non-zero value is $\tau \nu \kappa$, which is detected
by $\tau h_2 d_0$.
\end{proof}

\begin{lemma}
Extension of scalars takes elements detected by
$\rho^3 \tau^2 f_0$ to elements detected by $\tau^4 h_1 d_0$.
\end{lemma}

\begin{proof}
The long exact sequence of
Corollary \ref{cor:LES} gives a short exact sequence
\[
0 \map (\coker \rho)_{15,5} \map \pi^\C_{15,5} \map
(\ker \rho)_{15,6} \map 0.
\]
The group $\pi^\C_{15,5}$ is generated by an element
of order $32$, detected by $\tau^3 h_0^3 h_4$, and
an element of order $2$, detected by $\tau^4 h_1 d_0$.
Also $(\ker \rho)_{15,6}$ is generated by an element
of order $32$, detected by $\tau^2 h_0 \cdot h_0^3 h_4$.
It follows that $(\coker \rho)_{15,5}$ maps onto an
element of order $2$ that is detected by $\tau^4 h_1 d_0$.
\end{proof}

\section{Tables}\label{sec:appendix}

\begin{longtable}{llll}
\caption{Some values of the $\R$-motivic Mahowald invariant} \\
\toprule 
$s$ & $\alpha$ & $M^\R(\alpha)$ & indeterminacy \\
\midrule \endhead
\bottomrule \endfoot
\label{tab:R-Mahowald}
$0$ & $2^k$ & $\eta^k$ \\
$1$ & $\eta$ & $\nu$ & $2\nu$, $4 \nu$ \\
$2$ & $\eta^2$ & $\nu^2$ \\
$3$ & $\nu$ & $\sigma$ & $2 \sigma$, $4 \sigma$, $8 \sigma$ \\
$3$ & $2\nu$ & $\eta \sigma$ & $\epsilon$ \\
$3$ & $4 \nu$ & $\eta^2 \sigma$ & $\eta \epsilon$ \\
$6$ & $\nu^2$ & $\sigma^2$ & $\kappa$ \\
$7$ & $\sigma$ & $\tau \sigma^2$ \\
$7$ & $2 \sigma$ & $\eta_4$ & $\eta \rho_{15}$ \\
$7$ & $4 \sigma$ & $\eta \eta_4$ & $\eta^2 \rho_{15}$, $\nu \kappa$ \\
$7$ & $8 \sigma$ & $\eta^2 \eta_4$ & $\eta^3 \rho_{15}$ \\
$8$ & $\eta \sigma$ & $\nu_4$ & $2 \nu_4$, $4 \nu_4$ \\
$8$ & $\epsilon$ & $\sigmabar$ \\
$9$ & $\eta^2 \sigma$ & $\nu \nu_4$ & $\tau \eta \kappabar$ \\
$9$ & $\eta \epsilon$ & $\nu \sigmabar$ & $\tau \eta^2 \kappabar$ \\
$9$ & $\mu_9$ & $\nu \kappabar$ & $2 \nu \kappabar$, $4 \nu \kappabar$ \\
$10$ & $\eta \mu_9$ & $\nu \cdot \nu \kappabar$ \\
$11$ & $\zeta_{11}$ & $\tau \nu^2 \kappabar$ & $\eta^3 \rho_{23}$ \\
$11$ & $2 \zeta_{11}$ & $\{h_1 h_3 g\}$ & $\eta^5 \rho_{23}$ \\
$11$ & $4 \zeta_{11}$ & $\eta \{h_1 h_3 g\}$ & $\eta^6 \rho_{23}$ \\
\end{longtable}

\begin{longtable}{lllll}
\caption{$h_1$-periodic Bockstein differentials} \\
\toprule 
coweight & $(s,f,w)$ & $x$ & $d_r$ & $d_r(x)$ \\
\midrule \endfirsthead
\caption[]{$h_1$-periodic Bockstein differentials} \\
\toprule 
coweight & $(s,f,w)$ & $x$ & $d_r$ & $d_r(x)$ \\
\midrule \endhead
\bottomrule \endfoot
\label{tab:Bock-h1-local}
$4$ & $(9,5,5)$ & $P h_1$ & $d_3$ & $h_1^3 c_0$ \\
$7$ & $(16,7,9)$ & $P c_0$ & $d_3$ & $h_1^4 d_0$ \\
$8$ & $(17,9,9)$ & $P^2 h_1$ & $d_7$ & $h_1^6 e_0$ \\
$10$ & $(22,8,12)$ & $P d_0$ & $d_3$ & $h_1^2 c_0 d_0$ \\
$11$ & $(25,8,14)$ & $P e_0$ & $d_3$ & $h_1^2 c_0 e_0$ \\
$12$ & $(25,13,13)$ & $P^3 h_1$ & $d_3$ & $P^2 h_1^3 c_0$ \\
$13$ & $(30,11,17)$ & $P c_0 d_0$ & $d_3$ & $h_1^4 d_0^2$ \\
\end{longtable}

\begin{longtable}{lllll}
\caption{Bockstein differentials} \\
\toprule 
coweight & $(s,f,w)$ & $x$ & $d_r$ & $d_r(x)$ \\
\midrule \endfirsthead
\caption[]{Bockstein differentials} \\
\toprule 
coweight & $(s,f,w)$ & $x$ & $d_r$ & $d_r(x)$ \\
\midrule \endhead
\bottomrule \endfoot
\label{tab:Bock}
$1$ & $(0, 0, -1)$ & $\tau$ & $d_1$ & $h_0$ \\
$2$ & $(0, 0, -2)$ & $\tau^2$ & $d_2$ & $\tau h_1$ \\
$4$ & $(0, 0, -4)$ & $\tau^4$ & $d_4$ & $\tau^2 h_2$ \\
$4$ & $(1, 1, -3)$ & $\tau^4 h_1$ & $d_6$ & $\tau h_2^2$ \\
$4$ & $(2, 2, -2)$ & $\tau^4 h_1^2$ & $d_7$ & $c_0$ \\
$4$ & $(7, 4, 3)$ & $\tau h_0^3 h_3$ & $d_4$ & $h_1^2 c_0$ \\
$4$ & $(9, 5, 5)$ & $P h_1$ & $d_3$ & $h_1^3 c_0$ \\
$5$ & $(6, 2, 1)$ & $\tau^3 h_2^2$ & $d_3$ & $\tau c_0$ \\
$6$ & $(7, 4, 1)$ & $\tau^3 h_0^3 h_3$ & $d_3$ & $\tau P h_1$ \\
$6$ & $(9, 4, 3)$ & $\tau^3 h_1 c_0$ & $d_3$ & $P h_2$ \\
$7$ & $(8, 3, 1)$ & $\tau^4 c_0$ & $d_7$ & $d_0$ \\
$7$ & $(11, 5, 4)$ & $\tau^2 P h_2$ & $d_6$ & $h_1^2 d_0$ \\
$7$ & $(14, 6, 7)$ & $\tau h_0^2 d_0$ & $d_4$ & $h_1^3 d_0$ \\
$7$ & $(16, 7, 9)$ & $P c_0$ & $d_3$ & $h_1^4 d_0$ \\
$8$ & $(0, 0, -8)$ & $\tau^8$ & $d_8$ & $\tau^4 h_3$ \\
$8$ & $(2, 2, -6)$ & $\tau^8 h_1^2$ & $d_{13}$ & $\tau h_0 h_3^2$ \\
$8$ & $(3, 3, -5)$ & $\tau^8 h_1^3$ & $d_{15}$ & $e_0$ \\
$8$ & $(7, 4, -1)$ & $\tau^5 h_0^3 h_3$ & $d_{12}$ & $h_1 e_0$ \\
$8$ & $(9, 5, 1)$ & $\tau^4 P h_1$ & $d_{11}$ & $h_1^2 e_0$ \\
$8$ & $(15, 8, 7)$ & $\tau h_0^7 h_4$ & $d_8$ & $h_1^5 e_0$ \\
$8$ & $(17, 9, 9)$ & $P^2 h_1$ & $d_7$ & $h_1^6 e_0$ \\
$9$ & $(3, 1, -6)$ & $\tau^8 h_2$ & $d_{12}$ & $\tau^2 h_3^2$ \\
$9$ & $(14, 3, 5)$ & $\tau^3 h_0 h_3^2$ & $d_5$ & $f_0$ \\
$9$ & $(14, 6, 5)$ & $\tau^3 h_0^2 d_0$ & $d_3$ & $\tau P c_0$ \\
$9$ & $(20, 4, 11)$ & $\tau g$ & $d_1$ & $h_0 g$ \\
$10$ & $(6, 2, -4)$ & $\tau^8 h_2^2$ & $d_{14}$ & $\tau c_1$ \\
$10$ & $(9, 3, -1)$ & $\tau^7 h_1^2 h_3$ & $d_9$ & $\tau^2 e_0$ \\
$10$ & $(14, 4, 4)$ & $\tau^4 d_0$ & $d_5$ & $\tau^2 h_1 e_0$ \\
$10$ & $(15, 8, 5)$ & $\tau^3 h_0^7 h_4$ & $d_3$ & $\tau P^2 h_1$ \\
$10$ & $(17, 8, 7)$ & $\tau^3 P h_1 c_0$ & $d_3$ & $P^2 h_2$ \\
$10$ & $(20, 4, 10)$ & $\tau^2 g$ & $d_2$ & $\tau h_1 g$ \\
$10$ & $(22, 8, 12)$ & $P d_0$ & $d_3$ & $h_1^2 c_0 d_0$ \\
$11$ & $(8, 2, -3)$ & $\tau^8 h_1 h_3$ & $d_{12}$ & $\tau^2 c_1$ \\
$11$ & $(14, 3, 3)$ & $\tau^5 h_0 h_3^2$ & $d_5$ & $\tau^2 f_0$ \\
$11$ & $(17, 4, 6)$ & $\tau^4 e_0$ & $d_5$ & $\tau^2 h_1 g$ \\
$11$ & $(20, 6, 9)$ & $\tau^3 h_0 h_2 e_0$ & $d_6$ & $c_0 e_0$ \\
$11$ & $(23, 5, 12)$ & $\tau^2 h_2 g$ & $d_3$ & $h_1^2 h_4 c_0$ \\
$11$ & $(23, 7, 12)$ & $i$ & $d_4$ & $h_1 c_0 e_0$ \\
$11$ & $(25, 8, 14)$ & $P e_0$ & $d_3$ & $h_1^2 c_0 e_0$ \\
$12$ & $(7, 4, -5)$ & $\tau^9 h_0^3 h_3$ & $d_5$ & $\tau^6 P h_2$ \\
$12$ & $(9, 5, -3)$ & $\tau^8 P h_1$ & $d_6$ & $\tau^5 h_0^2 d_0$ \\
$12$ & $(10, 6, -2)$ & $\tau^8 P h_1^2$ & $d_7$ & $\tau^4 P c_0$ \\
$12$ & $(14, 2, 2)$ & $\tau^6 h_3^2$ & $d_6$ & $\tau^3 c_1$ \\
$12$ & $(15, 8, 3)$ & $\tau^5 h_0^7 h_4$ & $d_5$ & $\tau^2 P^2 h_2$ \\
$12$ & $(17, 9, 5)$ & $\tau^4 P^2 h_1$ & $d_6$ & $\tau P h_0^2 d_0$ \\
$12$ & $(18, 10, 6)$ & $\tau^4 P^2 h_1^2$ & $d_7$ & $P^2 c_0$ \\
$12$ & $(23, 12, 11)$ & $\tau h_0^5 i$ & $d_4$ & $P^2 h_1^2 c_0$ \\
$12$ & $(25, 13, 13)$ & $P^3 h_1$ & $d_3$ & $P^2 h_1^3 c_0$ \\
$13$ & $(14, 3, 1)$ & $\tau^7 h_0 h_3^2$ & $d_7$ & $\tau^4 g$ \\
$13$ & $(17, 4, 4)$ & $\tau^6 e_0$ & $d_5$ & $\tau^4 h_1 g$ \\
$13$ & $(18, 5, 5)$ & $\tau^6 h_1 e_0$ & $d_6$ & $\tau^3 h_0 h_2 g$ \\
$13$ & $(20, 6, 7)$ & $\tau^5 h_0 h_2 e_0$ & $d_7$ & $j$ \\
$13$ & $(22, 10, 9)$ & $\tau^3 P h_0^2 d_0$ & $d_3$ & $\tau P^2 c_0$ \\
$13$ & $(23, 7, 10)$ & $\tau^2 i$ & $d_6$ & $d_0^2$ \\
$13$ & $(25, 8, 12)$ & $\tau^2 P e_0$ & $d_5$ & $h_1 d_0^2$ \\
\end{longtable}

\begin{landscape}

\begin{longtable}{lllllll}
\caption{Some Massey products in $\Ext_\R$} \\
\toprule 
coweight & $(s,f,w)$ & bracket & contains & indeterminacy & proof & used in \\
\midrule \endhead
\bottomrule \endfoot
\label{tab:Massey}
3 & $(3,1,0)$ & $\an{\rho^2, \tau h_1, h_2}$ & $\tau^2 h_2$ & $\rho^4 h_3$ &
	$d_2(\tau^2) = \rho^2 \tau h_1$ & $\an{\rho^2, \tau \eta, \nu}$, Lemma \ref{lem:<rho^2, tau eta, nu_4>} \\
4 & $(8, 3, 4)$ & $\an{c_0, h_0, \rho}$ & $\tau c_0$ & 
	$\rho \tau h_1 \cdot h_1 h_3$ & $d_1(\tau) = \rho h_0$ & 
	$\an{\epsilon, \hsf, \rho}$ \\
$7$ & $(7, 1, 0)$ & $\an{\rho^4, \tau^2 h_2, h_3}$ & $\tau^4 h_3$ &
	$\rho^8 h_4$ & $d_4(\tau^4) = \rho^4 \tau^2 h_2$ &
	$\an{\rho^4, \tau^2 \nu, \sigma}$ \\
9 & $(21, 5, 12)$ & $\an{\tau h_1, h_1^4, h_4}$ & $h_2 f_0$ & $0$ &
	$\C$-motivic & Lemma \ref{lem:d2-t^2g} \\
9 & $(21,5,12)$ & $\an{\rho, h_2 e_0, h_1}$ & $h_2 f_0$ & $\rho^2 h_2 g$
	& $d_1(\tau g) = \rho h_2 e_0$ & $\an{\rho, \{h_2 e_0\}, \eta}$ \\
10 & $(18,4,8)$ & $\an{\tau^2 h_2, h_3, h_0^2 h_3}$ & $\tau^2 f_0$ &
	$\tau^2 h_2 \cdot h_0^2 h_4$, $\rho^5 h_4 c_0$ & $\C$-motivic & Lemma \ref{lem:t^2 f0 * h1} \\
10 & $(21,5,11)$ & $\an{\tau h_2^2, h_3, h_0^2 h_3}$ & $\tau^2 h_1 g$ &
	$\rho^3 h_1 h_4 c_0$ & $\C$-motivic & Lemma \ref{lem:t^2 f0 * h1} \\
11 & $(3,1,-8)$ & $\an{\rho^2, \tau^9 h_1, h_2}$ & $\tau^{10}h_2$ & 0 &
	$d_2(\tau^{10}) = \rho^2\tau^9 h_1$ & $\an{\rho^2, \tau^9 \eta, \nu}$ \\
11 & $(9,4, -2)$ & $\an{\tau h_1 \cdot \tau^5 c_0, \tau h_1, \rho^2}$
	& $h_1 \cdot \tau^8 c_0$ & 0 & $d_2(\tau^2) = \rho^2 \tau h_1$ & Lemma \ref{lem:t^8 h1 c0 * h1} \\
11 & $(11,5,0)$ & $\an{\rho^2, \tau^5 h_1, P h_2}$ & $\tau^6 P h_2$ &
	$\rho^{16} h_3 g$ & $d_2(\tau^6) = \rho^2\tau^5 h_1$ & 
	$\an{\rho^2, \tau^5 \eta, \zeta_{11}}$ \\
11 & $(14,6,3)$ & $\an{h_1, \tau^4 P h_2, \tau h_1}$ & $\tau^5 h_0^2 d_0$ &
	$0$ & $\C$-motivic & Lemma \ref{lem:t^8 h1 c0 * h1} \\
11 & $(17, 8, 6)$ & $\an{\tau h_1 \cdot \tau P c_0, \tau h_1, \rho^2}$ &
	$h_1 \cdot \tau^4 P c_0$ & $0$ & $d_2(\tau^2) = \rho^2 \tau h_1$ &
	Lemma \ref{lem:t^8 h1 c0 * h1} \\
11 & $(19,3,8)$ & $\an{\rho, h_0, \tau^2 c_1}$ & $\tau^3 c_1$ &
	$\rho^2 \tau^2 h_2 \cdot h_2 h_4$ & $d_1(\tau) = \rho h_0$ & 
	$\an{\rho, \hsf, \tau^2 \sigmabar}$ \\
11 & $(19, 3, 8)$ & $\an{\rho^2, \tau h_1, \tau c_1}$ & $\tau^3 c_1$ &
	$\rho^2 \tau^2 h_2 \cdot h_2 h_4$ & $d_2(\tau^2) = \rho^2 \tau h_1$ &
	Lemma \ref{lem:h1 * t^3 c1} \\
11 & $(19,9,8)$ & $\an{\rho^2,\tau h_1,P^2 h_2}$ & $\tau^2 P^2 h_2$ & 0 &
	$d_2(\tau^2) = \rho^2 \tau h_1$ & $\an{\rho^2, \tau \eta, \zeta_{19}}$ \\
11 & $(22,4,11)$ & $\an{\tau h_1, \tau c_1, h_1}$ & $h_2 \cdot \tau^2 c_1$ & $\rho h_4 \cdot \tau c_0$ & 
	$\C$-motivic & Lemma \ref{lem:h1 * t^3 c1} \\
11 & $(22,10,11)$ & $\an{h_1, P^2 h_2, \tau h_1}$ & $\tau P h_0^2 d_0$ &
	$0$ & $\C$-motivic & Lemma \ref{lem:t^8 h1 c0 * h1} \\
12 & $(20,4,8)$ & $\an{\rho, \tau^2 h_0, \rho, h_2 e_0}$ & $\tau^4 g$ & $\rho^2 h_2 \cdot \tau^3 c_1$ &
	$d_1(\tau^3) = \rho \tau^2 h_0$, & 
	$\an{\rho, \tau^2 \hsf, \rho, \{h_2 e_0\}}$
	\\&&&&&$d_1(\tau g) = \rho h_2 e_0$ \\
\end{longtable}

\end{landscape}

\begin{longtable}{lllll}
\caption{Hidden $h_0$ extensions in the $\rho$-Bockstein spectral sequence} \\
\toprule 
coweight & $(s,f,w)$ & source & target \\
\midrule \endhead
\bottomrule \endfoot
\label{tab:Bock-h0-extn}
1 & $(1, 1, 0)$ & $\tau h_1$ & $\rho \tau h_1^2$ 
\\3 & $(3, 3, 0)$ & $\tau^2 h_0^2 h_2$ & $\rho^6 h_1 c_0$ 
\\3 & $(7, 4, 4)$ & $h_0^3 h_3$ & $\rho^3 h_1^2 c_0$ 
\\4 & $(6, 2, 2)$ & $\tau^2 h_2^2$ & $\rho^2 \tau c_0$ 
\\4 & $(8, 3, 4)$ & $\tau c_0$ & $\rho \tau h_1 c_0$ 
\\5 & $(1, 1, -4)$ & $\tau^5 h_1$ & $\rho \tau^5 h_1^2$ 
\\5 & $(7, 4, 2)$ & $\tau^2 h_0^3 h_3$ & $\rho^2 \tau P h_1$ 
\\5 & $(9, 4, 4)$ & $\tau^2 h_1 c_0$ & $\rho^2 P h_2$ 
\\5 & $(9, 5, 4)$ & $\tau P h_1$ & $\rho \tau P h_1^2$ 
\\6 & $(6, 2, 0)$ & $\tau^4 h_2^2$ & $\rho^3 \tau^3 h_2^3$ 
\\6 & $(14, 6, 8)$ & $h_0^2 d_0$ & $\rho^3 h_1^3 d_0$ 
\\7 & $(3, 3, -4)$ & $\tau^6 h_0^2 h_2$ & $\rho^{14} e_0$ 
\\7 & $(7, 4, 0)$ & $\tau^4 h_0^3 h_3$ & $\rho^{11} h_1 e_0$ 
\\7 & $(11, 7, 4)$ & $\tau^2 P h_0^2 h_2$ & $\rho^{10} h_1^4 e_0$ 
\\7 & $(15, 8, 8)$ & $h_0^7 h_4$ & $\rho^7 h_1^5 e_0$ 
\\8 & $(8, 3, 0)$ & $\tau^5 c_0$ & $\rho \tau^5 h_1 c_0$
\\8 & $(14, 3, 6)$ & $\tau^2 h_0 h_3^2$ & $\rho^4 f_0$
\\8 & $(14, 6, 6)$ & $\tau^2 h_0^2 d_0$ & $\rho^2 \tau P c_0$
\\8 & $(16, 7, 8)$ & $\tau P c_0$ & $\rho \tau P h_1 c_0$
\\9 & $(1, 1, -8)$ & $\tau^9 h_1$ & $\rho \tau^9 h_1^2$
\\9 & $(7, 4, -2)$ & $\tau^6 h_0^3 h_3$ & $\rho^2 \tau^5 P h_1$
\\9 & $(9, 3, 0)$ & $\tau^6 h_1^2 h_3$ & $\rho^8 \tau^2 e_0$
\\9 & $(9, 4, 0)$ & $\tau^6 h_1 c_0$ & $\rho^2 \tau^4 P h_2$
\\9 & $(9, 5, 0)$ & $\tau^5 P h_1$ & $\rho \tau^5 P h_1^2$
\\9 & $(15, 8, 6)$ & $\tau^2 h_0^7 h_4$ & $\rho^2 \tau P^2 h_1$
\\9 & $(17, 8, 8)$ & $\tau^2 P h_1 c_0$ & $\rho^2 P^2 h_2$
\\9 & $(17, 9, 8)$ & $\tau P^2 h_1$ & $\rho \tau P^2 h_1^2$
\\10 & $(14, 3, 4)$ & $\tau^4 h_0 h_3^2$ & $\rho^4 \tau^2 f_0$
\\10 & $(18, 5, 8)$ & $\tau^2 h_0 f_0$ & $\rho^5 \tau h_2^2 e_0$
\\10 & $(20, 6, 10)$ & $\tau^2 h_0 h_2 e_0$ & $\rho^5 c_0 e_0$ 
\\11 & $(3, 3, -8)$ & $\tau^{10} h_0^2 h_2$ & $\rho^6 \tau^8 h_1 c_0$
\\11 & $(7, 4, -4)$ & $\tau^8 h_0^3 h_3$ & $\rho^4 \tau^6 P h_2$
\\11 & $(11, 7, 0)$ & $\tau^6 P h_0^2 h_2$ & $\rho^6 \tau^4 P h_1 c_0$
\\11 & $(15, 8, 4)$ & $\tau^4 h_0^7 h_4$ & $\rho^4 \tau^2 P^2 h_2$
\\11 & $(19, 3, 8)$ & $\tau^3 c_1$ & $\rho^3 \tau^2 h_2 c_1$
\\11 & $(19, 11, 8)$ & $\tau^2 P^2 h_0^2 h_2$ & $\rho^6 P^2 h_1 c_0$
\\11 & $(23, 12, 12)$ & $h_0^5 i$ & $\rho^3 P^2 h_1^2 c_0$
\\12 & $(6, 2, -6)$ & $\tau^{10} h_2^2$ & $\rho^2 \tau^9 c_0$
\\12 & $(8, 3, -4)$ & $\tau^9 c_0$ & $\rho \tau^9 h_1 c_0$
\\12 & $(14, 3, 2)$ & $\tau^6 h_0 h_3^2$ & $\rho^6 \tau^4 g$
\\12 & $(14, 6, 2)$ & $\tau^6 h_0^2 d_0$ & $\rho^2 \tau^5 P c_0$
\\12 & $(16, 7, 4)$ & $\tau^5 P c_0$ & $\rho \tau^5 P h_1 c_0$
\\12 & $(18, 5, 6)$ & $\tau^6 h_0 f_0$ & $\rho^5 \tau^3 h_2^2 e_0$
\\12 & $(20, 6, 8)$ & $\tau^4 h_0^2 g$ & $\rho^6 j$
\\12 & $(22, 10, 10)$ & $\tau^2 P h_0^2 d_0$ & $\rho^2 \tau P^2 c_0$
\\12 & $(24, 11, 12)$ & $\tau P^2 c_0$ & $\rho \tau P^2 h_1 c_0$
\\12 & $(26, 9, 14)$ & $h_0^2 j$ & $\rho^4 h_1^2 d_0^2$
\\
\end{longtable}

\begin{longtable}{lllll}
\caption{Hidden $h_1$ extensions in the $\rho$-Bockstein spectral sequence} \\
\toprule 
coweight & $(s,f,w)$ & source & target & proof \\
\midrule \endhead
\bottomrule \endfoot
\label{tab:Bock-h1-extn}
2 & $(0, 1, -2)$ & $\tau^2 h_0$ & $\rho \tau^2 h_1^2$ & 
\\3 & $(3, 1, 0)$ & $\tau^2 h_2$ & $\rho^2 \tau h_2^2$ & 
\\3 & $(6, 2, 3)$ & $\tau h_2^2$ & $\rho c_0$ & 
\\5 & $(9, 4, 4)$ & $\tau^2 h_1 c_0$ & $\rho P h_2$ & 
\\6 & $(0, 1, -6)$ & $\tau^6 h_0$ & $\rho \tau^6 h_1^2$ & 
\\6 & $(9, 3, 3)$ & $\tau^3 h_2^3$ & $\rho^4 d_0$ & Lemma \ref{lem:h1 t^3 h1^2 h3}
\\7 & $(14, 3, 7)$ & $\tau h_0 h_3^2$ & $\rho^2 e_0$ & 
\\9 & $(9, 3, 0)$ & $\tau^6 h_1^2 h_3$ & $\rho^7 \tau^2 e_0$ & 
\\9 & $(9, 4, 0)$ & $\tau^6 h_1 c_0$ & $\rho \tau^4 P h_2$ & 
\\9 & $(17, 8, 8)$ & $\tau^2 P h_1 c_0$ & $\rho P^2 h_2$ & 
\\9 & $(18, 5, 9)$ & $\tau^2 h_1 e_0$ & $\rho \tau h_2^2 d_0$ & 
\\10 & $(0, 1, -10)$ & $\tau^{10} h_0$ & $\rho \tau^{10} h_1^2$ & 
\\10 & $(14, 2, 4)$ & $\tau^4 h_3^2$ & $\rho^4 \tau^2 c_1$ & 
\\10 & $(18, 4, 8)$ & $\tau^2 f_0$ & $\rho^2 \tau^2 h_1 g$ & Lemma \ref{lem:t^2 f0 * h1}
\\10 & $(19, 3, 9)$ & $\tau^2 c_1$ & $\rho^2 \tau h_2 c_1$ &
\\11 & $(3,1,-8)$ & $\tau^{10} h_2$ & $\rho^2 \tau^9 h_2^2$ & 
\\11 & $(6,2,-5)$ & $\tau^9 h_2^2$ & $\rho \tau^8 c_0$ & 
\\11 & $(9,4,-2)$ & $\tau^8 h_1 c_0$ & $\rho \tau^6 P h_2$ & Lemma \ref{lem:t^8 h1 c0 * h1}
\\11 & $(11,5,0)$ & $\tau^6 P h_2$ & $\rho^2 \tau^5 h_0^2 d_0$ & Lemma
\ref{lem:t^8 h1 c0 * h1}
\\11 & $(14,6,3)$ & $\tau^5 h_0^2 d_0$ & $\rho \tau^4 P c_0$ & 
\\11 & $(17,8,6)$ & $\tau^4 P h_1 c_0$ & $\rho \tau^2 P^2 h_2$ & Lemma \ref{lem:t^8 h1 c0 * h1}
\\11 & $(19,3,8)$ & $\tau^3 c_1$ & $\rho^2 \tau^2 h_2 c_1$ & Lemma \ref{lem:h1 * t^3 c1}
\\11 & $(19,9,8)$ & $\tau^2 P^2 h_2$ & $\rho^2 \tau P h_0^2 d_0$ & Lemma
\ref{lem:t^8 h1 c0 * h1}
\\11 & $(22,10,11)$ & $\tau P h_0^2 d_0$ & $\rho P^2 c_0$ & 
\\12 & $(21,5,9)$ & $\tau^4 h_1 g$ & $\rho \tau^3 h_2^2 e_0$ & 
\\12 & $(22,9,10)$ & $\tau^2 P h_0 d_0$ & $\rho \tau^2 P h_1^2 d_0$ &
\\12 & $(23,6,11)$ & $\tau^3 h_2^2 e_0$ & $\rho^2 j$ & Lemma \ref{lem:t^3 h2^2 e0}
\\12 & $(26,7,14)$ & $j$ & $\rho d_0^2$ & Lemma \ref{lem:t^3 h2^2 e0}
\end{longtable}

\newpage

\begin{longtable}{llll}
\caption{Multiplicative generators of $\pi^\R_{*,*}$} \\
\toprule 
coweight & $(s,w)$ & element & detected by  \\
\midrule \endhead
\bottomrule \endfoot
\label{tab:pi-notation}
$0$ & $(-1,-1)$ & $\rho$ & $\rho$ \\
$0$ & $(0,0)$ & $\hsf$ & $h_0$  \\
$0$ & $(1,1)$ & $\eta$ & $h_1$  \\
$1$ & $(1,0)$ & $\tau \eta$ & $\tau h_1$ \\
$1$ & $(3,2)$ & $\nu$ & $h_2$ \\
$2$ & $(0,-2)$ & $\tau^2 \hsf$ & $\tau^2 h_0$ \\
$3$ & $(3,0)$ & $\tau^2 \nu$ & $\tau^2 h_2$ \\
$3$ & $(6,3)$ & $\tau \nu^2$ & $\tau h_2^2$ \\ 
$3$ & $(7,4)$ & $\sigma$ & $h_3$ \\
$3$ & $(8,5)$ & $\epsilon$ & $c_0$ \\
$4$ & $(0,-4)$ & $\tau^4 \hsf$ & $\tau^4 h_0$ \\ 
$4$ & $(8,4)$ & $\tau \epsilon$ & $\tau c_0$ \\ 
$5$ & $(1,-4)$ & $\tau^5 \eta$ & $\tau^5 h_1$ \\
$5$ & $(9,4)$ & $\tau \mu_9$ & $\tau P h_1$ \\
$5$ & $(11,6)$ & $\zeta_{11}$ & $P h_2$ \\
$6$ & $(0,-6)$ & $\tau^6 \hsf$ & $\tau^6 h_0$ \\ 
$6$ & $(14,8)$ & $\kappa$ & $d_0$ \\
$7$ & $(7,0)$ & $\tau^4 \sigma$ & $\tau^4 h_3$ \\ 
$7$ & $(11,4)$ & $\tau^2 \zeta_{11}$ & $\rho^6 e_0$ \\
$7$ & $(14,7)$ & $\tau \sigma^2$ & $\rho h_4$ \\ 
$7$ & $(15,8)$ & $\rho_{15}$ & $h_0^3 h_4$ \\ 
$7$ & $(16,9)$ & $\eta_4$ & $h_1 h_4$ \\
$8$ & $(0,-8)$ & $\tau^8 \hsf$ & $\tau^8 h_0$ \\ 
$8$ & $(8,0)$ & $\tau^5 \epsilon$ & $\tau^5 c_0$ \\ 
$8$ & $(14,6)$ & $\tau^2 \sigma^2$ & $\tau^2 h_3^2$ \\ 
$8$ & $(16,8)$ & $\tau \eta_4$ & $\tau h_1 \cdot h_4$ \\
$8$ & $(17,9)$ & $\tau \nu \kappa$  & $\rho f_0$ \\ 
$8$ & $(18,10)$ & $\nu_4$ & $h_2 h_4$ \\
$8$ & $(19,11)$ & $\bar{\sigma}$ & $c_1$ \\ 
$8$ & $(20,12)$ & $\{ h_2 e_0\}$ & $h_2 e_0$ \\
$9$ & $(1,-8)$ & $\tau^9 \eta$ & $\tau^9 h_1$ \\
$9$ & $(9,0)$ & $\tau^5 \mu_9$ & $\tau^5 P h_1$ \\ 
$9$ & $(11,2)$ & $\tau^4 \zeta_{11}$ & $\tau^4 P h_2$ \\ 
$9$ & $(15,6)$ & $\tau^3 \eta \kappa$ & $\rho^2 \tau^2 e_0$ \\ 
$9$ & $(17,8)$ & $\tau\mu_{17}$ & $\tau P^2 h_1$ \\
$9$ & $(19,10)$ & $\tau \sigmabar$ & $\tau c_1$ \\ 
$9$ & $(19, 10)$ & $\zeta_{19}$ & $P^2 h_2$ \\
$9$ & $(21,12)$ & $\tau \eta \kappabar$ & $h_2 f_0$ \\ 
$9$ & $(23,14)$ & $\nu \kappabar$ & $h_2 g$ \\ 
$10$ & $(0,-10)$ & $\tau^{10} \hsf$ & $\tau^{10} h_0$ \\ 
$10$ & $(15,5)$ & $\tau^4 \eta \kappa$ & $\rho^3 \tau^2 f_0$ \\ 
$10$ & $(18,8)$ & $\tau^2 \nu_4$ & $\tau^2 h_2 \cdot h_4$ \\ 
$10$ & $(19, 9)$ & $\tau^2 \sigmabar$ & $\tau^2 c_1$ \\
$10$ & $(20,10)$ & $\tau^2 \hsf \kappabar$ & $h_2 \cdot \tau^2 e_0$ \\ 
$10$ & $(21,11)$ & $\tau \nu \nu_4$ & $\tau h_2^2 \cdot h_4$ \\ 
$11$ & $(3,-8)$ & $\tau^{10} \nu$ & $\tau^{10} h_2$ \\ 
$11$ & $(6,-5)$ & $\tau^9 \nu^2$ & $\tau^9 h_2^2$ \\ 
$11$ & $(8,-3)$ & $\tau^8 \epsilon$ & $\tau^8 c_0$ \\ 
$11$ & $(11,0)$ & $\tau^6 \zeta_{11}$ & $\tau^6 P h_2$ \\ 
$11$ & $(15,4)$ & $\tau^4 \rho_{15}$ & $\tau^4 h_0^3 h_4$ \\ 
$11$ & $(17,6)$ & $\tau^4 \nu \kappa$ & $\tau^2 h_0 \cdot \tau^2 e_0$ \\ 
$11$ & $(19,8)$ & $\tau^3 \sigmabar$ & $\tau^3 c_1$ \\ 
$11$ & $(19,8)$ & $\tau^2 \zeta_{19}$ & $\tau^2 P^2 h_2$ \\ 
$11$ & $(23,12)$ & $\rho_{23}$ & $h_0^2 i$ \\ 
$11$ & $(26,15)$ & $\tau \nu^2 \kappabar$ & $\rho h_3 g$ \\ 
$11$ & $(28,17)$ & $\{h_1 h_3 g\}$ & $h_1 h_3 g$ \\ 
\end{longtable}

\begin{landscape}

\begin{longtable}{llllll}
\caption{Some Toda brackets in $\pi_{*,*}$} \\
\toprule 
coweight & $(s,w)$ & bracket & detected by & proof & used in \\
\midrule 
\label{tab:Toda}
3 & $(3,0)$ & $\an{\rho^2, \tau \eta, \nu}$ & $\tau^2 h_2$ &
$\an{\rho^2, \tau h_1, h_2}$ & Table \ref{tab:perm}
\\4 & $(8,4)$ & $\an{\epsilon, \hsf, \rho}$ & $\tau c_0$ &
	$\an{c_0, h_0, \rho}$ & Table \ref{tab:perm}
\\$7$ & $(7,0)$ & $\an{\rho^4, \tau^2 \nu, \sigma}$ & $\tau^4 h_3$ &
	$\an{\rho^4, \tau^2 h_2, h_3}$ & Table \ref{tab:perm}
\\$7$ & $(14,7)$ & $\an{\rho, \hsf, \sigma^2}$ & $\rho h_4$ &
$d_2(h_4) = h_0 h_3^2$ & Lemma \ref{lem:R-to-C-ph4}
\\8 & $(8,0)$ & $\an{\tau^5 \eta, \hsf \nu, \nu}$ & $\tau^5 c_0$ &
$\C$-motivic & Table \ref{tab:perm}
\\8 & $(14,6)$ & $\an{\tau^2 \nu, \sigma, \nu}$ & $\tau^2 h_3^2$ &
	$\C$-motivic & Table \ref{tab:perm}, Lemma \ref{lem:eta * t^2 h3^2}
\\8 & $(16,8)$ & $\an{\sigma^2, 2, \tau \eta}$ & $\tau h_1 \cdot h_4$ & 
	$d_2(h_4) = (h_0 + \rho h_1) h_3^2$ & Table \ref{tab:perm}
\\8 & $(16, 8)$ & $\an{\tau \mu_9, \hsf\nu, \nu}$ & $\tau P c_0$ & 
	$\C$-motivic & Table \ref{tab:perm}
\\$8$ & $(17,9)$ & $\langle \rho, \hsf, \nu \kappa \rangle$ & $\rho f_0$ &
$d_2(f_0) = h_0^2 e_0$ & Lemma \ref{lem:R-to-C-pf0}
\\8 & $(18,10)$ & $\an{\nu, \sigma, \hsf \sigma}$ & $h_2 h_4$ & $d_2(h_4) = h_0 h_3^2$ & Table \ref{tab:perm}
\\9 & $(15, 6)$ & $\an{\rho, \rho \tau \eta, \tau \eta \cdot \kappa}$ &
	$\rho^2 \tau^2 e_0$ & $d_2(\tau^2 e_0) = \tau^2 h_1^2 d_0$ &
	Table \ref{tab:perm}
\\9 & $(21, 12)$ & $\an{\rho, \{h_2 e_0\}, \eta}$ & $h_2 f_0$ & 
$\an{\rho, h_2 e_0, h_1}$ & Lemma \ref{lem:h0 * t h1 g}
\\9 & $(21, 13)$ & $\an{\{h_2 e_0\}, \eta, \hsf}$ &
	$c_0 d_0$ & $\C$-motivic & Lemma \ref{lem:h0 * t h1 g}
\\10 & $(18,8)$ & $\an{\rho^2, \tau \eta, \nu_4}$ & $\tau^2 h_2 \cdot h_4$ &
Lemma \ref{lem:<rho^2, tau eta, nu_4>} & Table \ref{tab:perm}
\\10 & $(19,9)$ & $\an{\tau^2 \nu, \eta \sigma, \sigma}$ & $\tau^2 c_1$ &
	$\C$-motivic & Table \ref{tab:perm}
\\11 & $(3,-8)$ & $\an{\rho^2, \tau^9 \eta, \nu}$ & $\tau^{10}h_2$ &
$\an{\rho^2, \tau^9 h_1, h_2}$ & Table \ref{tab:perm}
\\11 & $(11,0)$ & $\an{\rho^2, \tau^5 \eta, \zeta_{11}}$ & $\tau^6 P h_2$ & $\an{\rho^2, \tau^5 h_1, P h_2}$ & Table \ref{tab:perm}
\\11 & $(19,8)$ & $\an{\rho^2,\tau \eta,\zeta_{19}}$ & $\tau^2 P^2 h_2$ & $\an{\rho^2, \tau h_1, P^2 h_2}$ & Table \ref{tab:perm}
\\11 & $(19,8)$ & $\an{\rho, \hsf, \tau^2 \bar{\sigma}}$ & $\tau^3 c_1$ &
$\an{\rho, h_0, \tau^2 c_1}$ & Table \ref{tab:perm}
\\12 & $(8,-4)$ & $\an{\tau^9 \eta, \hsf \nu, \nu}$ & $\tau^9 c_0$ & $\C$-motivic & Table \ref{tab:perm}
\\12 & $(16,4)$ & $\an{\sigma^2, 2, \tau^5 \eta}$ & $\tau^5 h_1 \cdot h_4$ &
	$d_2(h_4) = (h_0 + \rho h_1) h_3^2$ & Table \ref{tab:perm}
\\12 & $(16,4)$ & $\an{\tau^5 \mu_9, \hsf \nu,\nu}$  & $\tau^5 P
c_0$ & $\C$-motivic & Table \ref{tab:perm}
\\12 & $(20,8)$ & $\an{\rho,\tau^2 \hsf,\rho,\{ h_2 e_0 \}}$ & $\tau^4 g$ &
$\an{\rho, \tau^2 h_0,\rho,h_2e_0}$ &
Table \ref{tab:perm}
\\12 & $(24, 12)$ & $\an{\tau \mu_{17}, \hsf \nu,\nu}$ & $\tau P^2 c_0$ &
$\C$-motivic & Table \ref{tab:perm}
\\\bottomrule
\end{longtable}

\end{landscape}

\begin{longtable}{llll}
\caption{Some permanent cycles in the $\R$-motivic Adams spectral sequence} \\
\toprule 
coweight & $(s,f,w)$ & element & proof \\
\midrule \endhead
\bottomrule \endfoot
\label{tab:perm}
3 & $(3,1,0)$ & $\tau^2 h_2$ & $\an{\rho^2, \tau \eta, \nu}$
\\4 & $(8,3,4)$ & $\tau c_0$ & $\an{\epsilon, \hsf, \rho}$
\\7 & $(7,1,0)$ & $\tau^4 h_3$ & $\an{\rho^4, \tau^2 \nu, \sigma}$
\\7 & $(11,4)$ & $\rho^6 e_0$ & Lemma \ref{lem:p^6 e0 perm}
\\8 & $(8,3,0)$ & $\tau^5 c_0$  & $\an{\tau^5 \eta, \hsf \nu, \nu}$
\\8 & $(14,6)$ & $\tau^2 h_3^2$ & $\an{\tau^2 \nu, \sigma, \nu}$
\\8 & $(16,7,8)$ & $\tau P c_0$ & $\an{\tau \mu_9, \hsf \nu, \nu}$
\\8 & $(16,2,8)$ & $\tau h_1 \cdot h_4$ & $\an{\sigma^2, 2, \tau \eta}$
\\8 & $(18,2,10)$ & $h_2 h_4$ & $\an{\nu, \sigma, \hsf \sigma}$
\\9 & $(15,4,6)$ & $\rho^2 \tau^2 e_0$ & 
	$\an{\rho, \rho \tau \eta, \tau \eta \cdot \kappa}$
\\10 & $(18,2,8)$ & $\tau^2 h_2 \cdot h_4$ & $\an{\rho^2, \tau \eta, \nu_4}$ 
\\10 & $(19,3,9)$ & $\tau^2 c_1$ & $\an{\tau^2 \nu, \eta \sigma, \sigma}$
\\11 & $(3,1,-8)$ & $\tau^{10}h_2$ & $\an{\rho^2,\tau^9 \eta,\nu}$
\\11 & $(11,5,0)$ & $\tau^6 P h_2$ & $\an{\rho^2,\tau^5 \eta, \zeta_{11}}$
\\11 & $(19,3,8)$ & $\tau^3 c_1$ & $\an{\rho, \hsf, \tau^2 \bar{\sigma}}$
\\11 & $(19,9,8)$ & $\tau^2 P^2 h_2$ & $\an{\rho^2,\tau \eta,\zeta_{19}}$
\\11 & $(23,4,12)$ & $h_4 \cdot \tau c_0$ & $\sigma \cdot \tau \eta_4$
\\12 & $(8,3,-4)$ & $\tau^9 c_0$ & $\an{\tau^9 \eta, \hsf \nu, \nu}$
\\12 & $(16,2,4)$ & $\tau^5 h_1 \cdot h_4$ & $\an{\sigma^2, 2, \tau^5 \eta}$
\\12 & $(16,7,4)$ & $\tau^5 P c_0$ & $\an{\tau^5 \mu_9, \hsf \nu,\nu}$
\\12 & $(20,4,8)$ & $\tau^4 g$ & $\an{\rho, \tau^2 h_0, \rho, h_2
e_0}$
\\12 & $(24,11,12)$ & $\tau P^2 c_0$ & $\an{\tau \mu_{17}, \hsf \nu,\nu}$
\end{longtable}

\begin{longtable}{lllll}
\caption{Adams $d_2$ differentials} \\
\toprule 
coweight & $(s,f,w)$ & $x$ & $d_2(x)$ & proof \\
\midrule \endhead
\bottomrule \endfoot
\label{tab:adams-d2}
7 & $(15,1,8)$ & $h_4$ & $h_0 h_3^2$  & classical
\\7 & $(17,4,10)$ & $e_0$ & $h_1^2 d_0$  & classical
\\7 & $(14, 3, 7)$ & $\tau h_0 h_3^2$ & $\rho^2 h_1 d_0$ & Lemma \ref{lem:d2(th0h3^2)}
\\8 & $(18, 4, 10)$ & $f_0$ & $h_0^2 e_0$  & Lemma \ref{lem:d2(f_0)}
\\9 & $(17, 4, 8)$ & $\tau^2 e_0$ & $(\tau h_1)^2 d_0$ & classical 
\\10 & $(18, 4, 8)$ & $\tau^2 f_0$ & $\tau^2 h_0^2 e_0 + \rho^3 \tau h_2^2 \cdot d_0$ & Lemma \ref{lem:d2(t^2 f0)}
\\10 & $(21, 5, 11)$ & $\tau^2 h_1 g$ & $\rho^2 c_0 d_0$ & Lemma \ref{d_2(t^2 h1 g)} 
\\ 11 & $(23, 8, 12)$ & $h_0 i$ & $P h_0^2 d_0$ & classical
\\ 11 & $(27, 5, 16)$ & $h_3 g$ & $h_1^3 h_4 c_0$ & $\C$-motivic
\\ 12 & $(26, 7, 14)$ & $j$ & $P h_2 \cdot d_0$ & classical
\end{longtable}

\newpage

\begin{longtable}{lllll}
\caption{Adams $d_3$ differentials} \\
\toprule 
coweight & $(s,f,w)$ & $x$ & $d_r(x)$ & proof \\
\midrule \endhead
\bottomrule \endfoot
\label{tab:higher-adams}
7 & $(15, 2, 8)$ & $h_0h_4$ & $h_0d_0 + \rho h_1 d_0$ & Lemma \ref{lem:h0 h4}
\\12 & $(23, 6, 11)$ & $\tau h_2^2 \cdot \tau^2 e_0$ & $\rho \tau P h_1 \cdot d_0$ & Lemma \ref{lem:p j}
\\12 & $(25,7,13)$ & $c_0 \cdot \tau^2 e_0$ & $\tau P h_1 \cdot h_1 d_0$ & Lemma \ref{lem:p j}
\end{longtable}

\begin{longtable}{lllll}
\caption{Hidden $\rho$ extensions in the $\R$-motivic 
Adams spectral sequence} \\
\toprule 
coweight & $(s,f,w)$ & source & target \\
\midrule \endhead
\bottomrule \endfoot
\label{tab:Adams-rho-extn}
$7$ & $(15, 4, 8)$ & $h_0^3 h_4$ & $\rho^4 h_1 e_0$ \\
$7$ & $(17,5,10)$ & $h_2 d_0$ & $\tau h_1 \cdot h_1 d_0$ \\
$8$ & $(15,2,7)$ & $\rho \tau h_1 \cdot h_4$ & $h_0 \cdot \tau^2 h_3^2$\\
$8$ & $(15,4,7)$ & $\rho^3 f_0$ & $\tau^2 h_0 \cdot d_0$ \\
10 & $(15,2,5)$ & $\rho^3 \tau^2 h_2 \cdot h_4$ & 
	$\tau^4 h_3 \cdot h_0 h_3$ \\
10 & $(15,4,5)$ & $\rho^3 \tau^2 f_0$ & $\tau^4 h_0 \cdot d_0$ \\
10 & $(23,8,13)$ & $h_1 \cdot \tau c_0 \cdot d_0$ & $P h_0 d_0$ \\
11 & $(15,4,4)$ & $\tau^4 h_0 \cdot h_0^2 h_4$ & $\tau^5 h_0^2 d_0$ \\
11 & $(17,5,6)$ & $\tau^2 h_0 \cdot \tau^2 e_0$ & $\tau^5 h_1 \cdot h_1 d_0$ \\
11 & $(18,5,7)$ & $\rho^3 f_0 \cdot \tau^2 h_2$ & $h_0 \cdot \tau^2 h_0 \cdot \tau^2 e_0$ \\
11 & $(23,9,12)$ & $h_0^2 i$ & $\tau P h_0^2 d_0$
\end{longtable}

\begin{longtable}{llll}
\caption{Hidden $\hsf$ extensions in the 
$\R$-motivic Adams spectral sequence} \\
\toprule 
coweight & $(s,w)$ & source & target  \\
\midrule \endhead
\bottomrule \endfoot
\label{tab:Adams-h0-extn}
7 & $(11,4)$ & $\rho^6 e_0$  &  $\tau^2 h_0\cdot Ph_2$  
\\9 & $(21, 12)$ & $h_2 f_0$ & $\rho c_0 d_0$ 
\\9 & $(23, 14)$ & $h_0 h_2 g$ & $h_1 c_0 d_0$ 
\\10 & $(22,12)$ & $\tau c_0 \cdot d_0$ & $Ph_0 d_0$ 
\\11 & $(23,12)$ & $\tau^2 h_0 \cdot h_2 g$ & $\tau P h_1 \cdot d_0$
\end{longtable}

\begin{longtable}{llll}
\caption{Hidden $\eta$ extensions in the $\R$-motivic Adams
spectral sequence} \\
\toprule 
coweight & $(s,f,w)$ & source & target \\
\midrule \endhead
\bottomrule \endfoot
\label{tab:Adams-eta-extn}
$7$ & $(15, 4, 8)$ & $h_0^3 h_4$ & $\rho^3 h_1^2 e_0$  
\\$9$ & $(15, 5, 6)$ & $\tau^2 h_0 \cdot h_0^3 h_4$ &
	$\rho \tau h_1 \cdot \tau P c_0$
\\9 & $(21,5,12)$ & $h_2 f_0$ & $c_0 d_0$
\\10 & $(20,5,10)$ & $h_2 \cdot \tau^2 e_0$ & $\rho \tau c_0 \cdot d_0$
\\10 & $(21,7,11)$ & $\rho \tau c_0 \cdot d_0$ & $P h_0 d_0$
\\11 & $(15,4,4)$ & $\tau^4 h_0 \cdot h_0^2 h_4$ & $\tau^4 P c_0$ 
\\11 & $(23,9,12)$ & $h_0^2 i$ & $P^2c_0$
\end{longtable}

\begin{longtable}{lllll}
\caption{Hidden values of extension by scalars} \\
\toprule 
coweight & $(s,f,w)$ & source & target \\
\midrule \endhead
\bottomrule \endfoot
\label{tab:R-to-C}
$7$ & $(11, 4, 4)$ & $\rho^6 e_0$ & $\tau^2 P h_2$
\\$7$ & $(14, 1, 7)$ & $\rho h_4$ & $\tau h_3^2$
\\$7$ & $(16+k, 6+k, 9+k)$ & $\rho^3 h_1^{k+2} e_0$ & $P h_1^k c_0$
\\$8$ & $(17, 4, 9)$ & $\rho f_0$ & $\tau h_2 d_0$
\\9 & $(15,4,6)$ & $\rho^2\tau^2 e_0$ & $\tau^3 h_1 d_0$
\\10 & $(15, 4, 5)$ & $\rho^3 \tau^2 f_0$ & $\tau^4 h_1 d_0$
\\10 & $(22, 7, 12)$ & $\tau c_0 \cdot d_0$ & $P d_0$
\\10 & $(23, 8, 13)$ & $h_1 \cdot \tau c_0 \cdot d_0$ & $P h_1 d_0$
\\11 & $(20, 5, 9)$ & $\tau^2 h_2 \cdot \rho f_0$ & $\tau^3 h_0^2 g$ 
\\11 & $(26,5,15)$ & $\rho h_3 g$ & $\tau h_2^2 g$
\\
\end{longtable}

\end{document}